\documentclass[12pt, reqno]{amsart}
\usepackage[english]{babel}
\usepackage{amsthm}
\usepackage{amssymb}
\usepackage[abbrev]{amsrefs}
\usepackage{yhmath}
\usepackage[usenames]{color}
\usepackage{bm}
\usepackage{enumitem}
\usepackage{hyperref}
\AtBeginDocument{\def\MR#1{}}
\newtheorem{thm}{}[section]
\newtheorem{theorem}[thm]{Theorem}
\newtheorem{corollary}[thm]{Corollary}
\newtheorem{lemma}[thm]{Lemma}
\newtheorem{proposition}[thm]{Proposition}
\theoremstyle{definition}
\newtheorem{definition}[thm]{Definition}
\theoremstyle{remark}

\numberwithin{equation}{section}
\allowdisplaybreaks
\renewcommand{\Re}{\operatorname{Re}}
\renewcommand{\Im}{\operatorname{Im}}
\DeclareMathOperator*{\Rad}{Rad}
\DeclareMathOperator*{\Ave}{Ave}

\newcommand{\enangle}[1]{\left\langle#1\right\rangle}
\newcommand{\abs}[1]{\left\lvert#1\right\rvert}
\newcommand{\norm}[1]{\left\lVert#1\right\rVert}
\newcommand{\enpar}[1]{\left(#1\right)}
\newcommand{\enbrace}[1]{\left\lbrace#1\right\rbrace}
\newcommand{\floor}[1]{\left\lfloor#1\right\rfloor}

\newcommand{\Fou}{\ensuremath{\mathcal{F}}}

\newcommand{\Jt}{\ensuremath{\mathcal{J}}}
\newcommand{\Ht}{\ensuremath{\mathcal{H}}}
\newcommand{\At}{\ensuremath{\mathcal{A}}}
\newcommand{\Dt}{\ensuremath{\mathcal{D}}}
\newcommand{\TT}{\ensuremath{\mathbb{T}}}

\newcommand{\HL}{\ensuremath{\mathring{H}}}
\newcommand{\LL}{\ensuremath{\mathcal{L}}}

\newcommand{\HH}{\ensuremath{\mathbb{H}}}
\newcommand{\NN}{\ensuremath{\mathbb{N}}}
\newcommand{\Nt}{\ensuremath{\mathcal{N}}}
\newcommand{\Mt}{\ensuremath{\mathcal{M}}}
\newcommand{\Ts}{\ensuremath{\mathcal{T}}}
\newcommand{\ZZ}{\ensuremath{\mathbb{Z}}}
\newcommand{\RR}{\ensuremath{\mathbb{R}}}
\newcommand{\FF}{\ensuremath{\mathbb{F}}}
\newcommand{\DD}{\ensuremath{\mathbb{D}}}
\newcommand{\CC}{\ensuremath{\mathbb{C}}}
\newcommand{\Ft}{\ensuremath{\mathcal{F}}}
\newcommand{\TL}{\ensuremath{\mathcal{T}}}
\newcommand{\St}{\ensuremath{\mathcal{S}}}

\newcommand{\XX}{\ensuremath{{\mathbb{X}}}}
\newcommand{\YY}{\ensuremath{{\mathbb{Y}}}}
\newcommand{\xx}{\ensuremath{\bm x}}
\newcommand{\yy}{\ensuremath{\bm y}}
\newcommand{\zz}{\ensuremath{\bm z}}
\newcommand{\ee}{\ensuremath{\bm e}}

\newcommand{\EE}{\ensuremath{\mathbb{E}}}
\newcommand{\Et}{\ensuremath{\mathcal{E}}}

\newcommand{\Ct}{\ensuremath{\mathcal{C}}}
\newcommand{\Pt}{\ensuremath{\mathcal{P}}}
\newcommand{\Gt}{\ensuremath{\mathcal{G}}}
\newcommand{\Bt}{\ensuremath{\mathcal{B}}}
\newcommand{\BB}{\ensuremath{\mathbb{B}}}
\newcommand{\XB}{\ensuremath{\mathcal{X}}}
\newcommand{\YB}{\ensuremath{\mathcal{Y}}}
\newcommand{\ZB}{\ensuremath{\mathcal{Z}}}
\newcommand{\Id}{\ensuremath{\mathrm{Id}}}

\newcommand{\supp}{\operatorname{supp}}

\title[Isomorphisms between vector-valued $H_{p}$-spaces for $0<p\le 1$]{Isomorphisms between vector-valued $\bm{H_{p}}$-spaces for $\bm{0<p\le 1}$ and uniqueness of unconditional structure}
\author[F. Albiac]{F. Albiac}
\address{Department of Mathematics, Statistics and Computer Sciences, and InaMat2\\ Universidad P\'ublica de Navarra\\
Pamplona 31006\\ Spain}
\email{fernando.albiac@unavarra.es}

\author[J. L. Ansorena]{J. L. Ansorena}
\address{Department of Mathematics and Computer Sciences\\
Universidad de La Rioja\\
Logro\~no 26004\\ Spain}
\email{joseluis.ansorena@unirioja.es}
\begin{document}
\begin{abstract}
The aim of this paper is twofold. On the one hand, we manage to identify Banach-valued Hardy spaces of analytic functions over the disc $\DD$ with other classes of Hardy spaces, thus complementing the existing literature on the subject. On the other hand, we develop new techniques that allow us to prove that certain Hilbert-valued atomic lattices have a unique unconditional basis, up to normalization, equivalence and permutation. Combining both lines of action we show that that $H_p(\DD,\ell_2)$ for $0<p<1$ has a unique atomic lattice structure. The proof of this result relies on the validity of some new lattice estimates for non-locally convex spaces which hold an independent interest.
\end{abstract}
\thanks{The authors acknowledge the support of the Spanish Ministry for Science and Innovation under Grant PID2022-138342NB-I00 for \emph{Functional Analysis Techniques in Approximation Theory and Applications} (TAFPAA)}
\maketitle
\section{Introduction}\noindent
Most classical spaces underlying the research in analysis and operator theory, such as the sequence spaces $\ell_{p}$, the function spaces $L_{p}$, the Hardy spaces $H_{p}$, or the Lorentz sequence spaces $d(w,p)$, are defined in terms of a parameter which serves to identify the members inside each family. Non-locally convex spaces show up very naturally in different areas when there is no reason to stop when the parameter $p$ is $1$.

Historically, much of the early impetus do develop a theory of non-locally convex spaces came from work in the Hardy spaces $H_{p}$. Because there is no clear intrinsic reason to restrict attention to the case $p\ge 1$, the theory of the spaces $H_{p}$ for $p<1$ has been studied in some detail, and so these provide some of the most interesting and best understood examples of non-locally convex spaces.

The researchers interested in Hardy spaces divide in two almost disjoint groups; more than half of them are solely interested in obtaining or improving bounds for operators whose domain or target space is $H_p$, while the other ones have a preference for structure theory. This article is devoted to the latter aspect, and to be more specific to the study of the uniqueness of unconditional structure of the non-locally convex vector-valued Hardy spaces.

Members of the Hardy class $H_{p}$ for $0<p<\infty$ may appear under different forms, to the extent that sometimes it can be difficult to realize they are the same space under disguise. The most classical $H_{p}$ space is the space $H_{p}(\DD)$ of all analytic functions $f$ on the unit disc in the complex plane such that
\[
\norm{ f}_{p}=\sup\limits_{0<r<1}\enpar{\frac{1}{2\pi}\int_{-\pi}^{\pi}\abs{f(re^{i\theta})}^{p}\,d\theta}^{1/p}<\infty.
\]
For $0<p\le 1$, $H_{p}(\TT)$ is the space of functions $f$ on $\TT$ such that $f=\sum_{n=1}^{\infty}\lambda_{n}\, a_{n}$ with $a_{n}$ $p$-atoms and $N:=\sum_{n=1}^{\infty}\abs{\lambda_{n}}^{p}<\infty$, equipped with the $p$-norm that provides the infimum value of $N^{1/p}$. It is known that any function in $H_p(\DD)$ has boundary values, and the space consisting of the real part of such boundary functions is $H_p(\TT)$. Besides, the norm of $f$ in $H_p(\TT)$ is equivalent to
\[
\enpar{\int_{-\pi}^\pi \sup_{0<r<1} \abs{P_r*f}^p}^{1/p},
\]
that is, the $L_p$-norm of the maximal function built via the Poisson kernel $(P_r)_{0<r<1}$.

Another class of $H_{p}$ spaces is that of martingale $H_{p}$ spaces. We will limit our attention to dyadic martingales only. A function $f$ belongs to $H_{p}^{\delta}([0,1))$, the dyadic $H_{p}$ space, if its maximal function $f^*$ satisfies
\[
\enpar{\int_{0}^{1} \enpar{f^{\ast}(t)}^{p}\,dt}^{1/p}<\infty.
\]
It is known that a function $f\in H_{p}^{\delta}([0,1))$, $0<p<\infty$, can be represented as a series $f=\sum_{n=1}^{\infty}\lambda_{n}\, h_{n}$, where $(h_{n})_{n=0}^{\infty}$ denotes the Haar system, and the norm of $f$ in $H_{p}^{\delta}([0,1))$ is comparable to
\[
\enpar{\int_{0}^{1}\left(\sum_{n=0}^{\infty} \abs{\lambda_{n}}^{2}\abs{h_{n}}^{2}\right)^{p/2}}^{1/p}.
\]
Maurey \cite{Maurey1980} proved that $H_{1}(\DD)$ and $H_{1}^{\delta}([0,1))$ are isomorphic. Since $H_{1}^{\delta}([0,1))$ possesses the Haar system as an unconditional basis it follows that $H_{1}(\DD)$ has an unconditional basis too, which settled an important open problem in the theory. However, the methods in Maurey's paper do not provide an explicit unconditional basis of $H_{1}(\DD)$. Four years later, Wojtaszczyk gave in \cite{Woj1984} a construction of an unconditional basis in $H_{p}(\DD)$ for $0<p\le 1$ and used the bases he constructed to give explicit isomorphisms between various $H_{p}$ spaces; most notably he proved that $H_{p}(\DD)$ is isomorphic to the space $H_{p}^{\delta}([0,1))$.

All the above Hardy spaces and norms have vector-valued counterparts but to extend to the Banach-valued setting the identification between any two of those spaces or the equivalence of norms is far from trivial. In \cite{Blasco1988} Blasco dealt with the problem of finding conditions on a Banach space $\XX$ so that one can identify the Hardy space $H_{p}( \DD,\XX)$ of $\XX$-valued, analytic functions on the disc $\DD$, with some space of boundary-values functions. He showed that, via the Poisson integral, $H_{p}( \DD,\XX)$ can be identified with a certain space of measures on the circle $\TT$ for $1\le p <\infty$. Under the assumption that $\XX$ has the analytic Radon-Nikod\'ym property (ARNP for short) such measures are representable as functions. The corresponding result for the Hardy spaces of harmonic functions required the Radon-Nikod\'ym property (RNP for short) instead of the ARNP \cite{BGC1987}. Moreover, in \cite{Blasco1988} it is shown that for $0<p<\infty$, the spaces $H_{p}(\DD, \XX)$ can be characterised using the maximal function if and only if $\XX$ is an unconditional martingale difference space (UMD space for short).

Continuing in this spirit, our first task in this note will be to identify the various classes of vector-valued non-locally convex Hardy spaces that appear in the literature depending on the context. This equivalence, combined with some specific new techniques that we develop to show the uniqueness of unconditional basis in vector-valued quasi-Banach spaces, will allow us to establish our main result on uniqueness of unconditional structure for Hardy spaces.

\begin{theorem}\label{thm:main}
Suppose $0<p<1$. Then the vector-valued Hardy space $H_p(\DD, \ell_{2})$ has a unique normalized unconditional basis, up to equivalence and permutation.
\end{theorem}

To put this result in context, we recall that $c_{0}$, $\ell_{1}$, and $\ell_{2}$ are the only Banach spaces which have a unique normalized unconditional basis up to equivalence. This means that if $\XX$ is one of those spaces and if $(\xx_{n})_{n=1}^{\infty}$ and $(\yy_{n})_{n=1}^{\infty}$ are two normalized unconditional bases of $\XX$ then the mapping $T$ defined by $T(\xx_{n})=\yy_{n}$ for all $n\in\NN$ extends to a linear automorphism of $\XX$. Any other Banach space with an unconditional basis fails to have this property. However, in the wider class of quasi-Banach spaces we find many other examples of spaces with that property, including the spaces $\ell_{p}$ for $0<p<1$, some Orlicz sequence spaces, and some Lorentz sequence spaces, to name but a few.

If a Banach space has, up to equivalence, a unique normalized unconditional basis then the basis must be equivalent to all its permutations and so it has to be symmetric. This remark, as well as the fact that an unconditional basis by its very definition does not have to come with a prescribed order, shows that for unconditional bases which are not symmetric, the most natural equivalence property is that of equivalence up to permutations. It is therefore interesting to discover if a space $\XX$ with a unconditional basis $(\xx_{n})_{n\in\Nt}$ has a uniqueness unconditional basis in this sense; that is, assuming that $(\xx_{n})_{n\in\Nt}$ is normalized, whenever $(\yy_m)_{m\in\Mt}$ is another normalized unconditional basis of $\XX$ there is a bijection $\pi\colon\Nt\to\Mt$ so that $(\yy_{\pi(n)})_{n\in\Nt}$ is equivalent to $(\xx_{n})_{n\in\Nt}$.

Bourgain et al.\ studied in \cite{BCLT1985} the problem of the uniqueness of unconditional basis in the vector-valued versions of the three Banach spaces with a unique unconditional basis. They showed that the spaces $c_{0}(\ell_{1})$, $\ell_{1}(c_{0})$, $c_{0}(\ell_{2})$, and $\ell_{1}(\ell_{2})$ have the property, whereas $\ell_{2}(\ell_{1})$ and $\ell_{2}(c_{0})$ do not. Many of the questions the authors formulated in their 1985 \emph{Memoir} remain open as of today. They conjectured that if a Banach space $\XX$ has a unique unconditional basis then so does the iterated copy of $\XX$ in the sense of one of the spaces $c_{0}$, $\ell_{1}$, and $\ell_{2}$. This conjecture was disproved in the general case in 1999 by Casazza and Kalton, who showed that Tsirelson's space $\Ts$ has a unique unconditional basis while $c_{0}(\Ts)$ does not (\cite{CasKal1999}). We emphasize that the anti-Euclidean nature of the spaces involved (that is, the lack of uniformly complemented copies of finite-dimensional Hilbert spaces) played a key role in the achievement of those results.

Casazza and Kalton's work gave thus continuity to a research topic that was central in Banach space theory in the 1960's and 1970's, but that was interrupted after the \emph{Memoir}. Perhaps the researchers felt discouraged to put effort into a subject that required the discovery of novel tools in order to make headway, with little hope for attaining a satisfactory classification of the Banach spaces with a unique unconditional basis.

At the same time, the positive results on uniqueness of unconditional basis obtained in the context of non-locally convex quasi-Banach spaces motivated further study with a number of authors contributing to the development of a coherent theory. An important advance was the paper \cite{KLW1990} by Kalton et al.\, followed by the work of Ler\'anoz \cite{Leranoz1992}, who proved that $c_{0}(\ell_{p})$ has a unique unconditional basis for all $0<p<1$. Subsequently, it was proved that the spaces $\ell_{p}(\ell_{2})$, $\ell_{p}(\ell_{1})$, and $\ell_{1}(\ell_{p})$ for all $0<p<1$ (\cites{AlbiacLeranoz2002, AKL2004}) also do.

For Hardy spaces, Wojtaszczyk proved that the Haar system is the unique unconditional basis of $H_{p}^{\delta}([0,1))$ for $0<p<1$. However, the situation is rather different when $p=1$. Indeed, to see (from a more modern approach) that $H_{1}(\DD)$ fails to have a unique unconditional basis let us first notice that the Haar system is a democratic basis in $H_{1}^{\delta}([0,1))$. On the other hand, the space $\ell_{2}$ is complemented in $H_{1}(\DD)$, and so the direct sum of the canonical $\ell_{2}$-basis and the Haar basis cannot be equivalent to a permutation of the Haar system because the latter is not democratic. Hence the question naturally arises of what can be said about the uniqueness of unconditional structure of $\XX$-valued Hardy spaces $H_{p}(\XX)$ for $p<1$, an aspect of the structure of these spaces which, as of today, is missing from the literature and which we tackle in Theorem~\ref{thm:main} in the particular case that $\XX=\ell_2$.

Despite the fact that $H_{p}(\ell_2)$ is the more natural vector-valued Hardy space, it has to be observed that is a convoluted breed of an Euclidean component (a separable Hilbert space) with an anti-Euclidean component that contains $\ell_2$ (the scalar-valued Hardy space $H_p$). The authors of \cite{AAW2022} succeeded in splitting any unconditional basis of $H_p\oplus\ell_2$ into an Euclidean part and an anti-Euclidean one, thereby proving that $H_p\oplus\ell_2$ has a unique unconditional basis. However this splitting technique, whose origin dates back to \cite{EdWo1976}, is not fitted for vector-valued spaces.To make things work requires a wide repertoire of techniques that includes lattice estimates in the lack of local convexity (see Section~\ref{sect:QBLE}), constructive function theory, martingale theory, harmonic analysis (see Section~\ref{sect:VVHE}), and unconditional basis theory (see Sections~\ref{sect:UBQB} and \ref{sect:AAMethod}). In Section~\ref{sect:HVLUTAP} we develop specific methods for studying unconditional bases of Hilbert-valued spaces which, combined with the previously developed machinery, will lead us to our goal. Throughout this article we use standard facts and notation in Banach space theory as can be found in \cites{AlbiacKalton2016}. Other more specialized terminology will be introduced in the context where it is needed.

\section{Quasi-Banach lattice estimates}\label{sect:QBLE}\noindent
Let $\XX$ be a vector space over the real or complex field $\mathbb F$. A {\it quasi-norm} on $\XX$ is a map $\Vert \cdot
\Vert \colon \XX\to [0,+\infty)$ satisfying
\begin{enumerate}[label=(\roman*), leftmargin=*, widest=iii]
\item $\norm{ f } >0$ if $f\not=0$;

\item $\norm{ \lambda \, f} = \abs{\lambda} \norm{ f }$ for all $f\in\XX$ and all $\lambda\in\FF$; and

\item\label{it:QN} $\norm{ f_{1}+f_{2} } \le C\enpar{\norm{ f_{1} } +\norm{ f_{2}} }$ for all $f_{1},
f_{2}\in\XX$,
\end{enumerate}
where $C\ge 1$ is a constant independent of $f_{1}$ and $f_{2}$. The optimal constant $C$ in \ref{it:QN} is called the \emph{modulus of concavity} of $\XX$. If $C=1$, $\norm{ \cdot}$ is a norm. A quasi-norm defines a metrizable vector topology on $\XX$ for which the family
\[
\enbrace{ f\in \XX\colon\norm{f} <\varepsilon}, \quad \varepsilon>0,
\]
is a base of neighbourhoods of zero. If such topology is complete then we say that $(\XX,\Vert \cdot\Vert)$ is a \emph{quasi-Banach space}.

A quasi-Banach space $\XX$ is said to be \emph{locally $r$-convex}, $0<r<\infty$, if there is a constant $C\ge 1$ such that
\begin{equation}\label{eq:pconvex}
\frac{1}{C} \norm{ \sum_{j\in J} f_j}\le N_r(\varphi):= \enpar{\sum_{j\in J} \norm{f_j}^r}^{1/r}.
\end{equation}
for every finite family $\varphi=(f_j)_{j\in J}$ in $\XX$. If a nontrivial quasi-Banach space is $r$-convex, then $r\le 1$. Conversely, the Aoki-Rolewicz theorem \cites{Aoki1942,Rolewicz1957} states that any quasi-Banach space is locally $r$-convex for some $r\in(0,1]$. Any $r$-convex quasi-Banach $\XX$ becomes a $r$-Banach space under a suitable re-norming, i.e., $\XX$ can be endowed with an equivalent quasi-norm fulfilling \eqref{eq:pconvex} with $C=1$. Such quasi-norms are called $r$-norms.

A partially ordered quasi-Banach space $(\LL,\norm{\cdot})$ over the reals is called a \emph{quasi-Banach lattice} provided
\begin{enumerate}[label=(\roman*), leftmargin=*, widest=iii]
\item $x+z\le y+z$ whenever $x\le y$, for every $x$, $y$ and $z \in \LL$;
\item $0\le \lambda\, x$ for every $0\le x$ in $\LL$ and every $\lambda\in [0,\infty)$;
\item any two elements $x$, $y$ in $\LL$, have a least upper bound $x\vee y$; and
\item if we define the \emph{absolute value} of an $x\in \LL$ as $\abs{x} = x\vee (-x)$ then $\norm{ x}\le \norm{ y}$ whenever $\abs{x}\le \abs{y}$.
\end{enumerate}

The lattice structure on $\LL$ induces a different notion of convexity. To formulate it we need to use the \emph{lattice functional calculus}, which is described e.g. in \cite{LinTza1979}*{Theorem 1.d.1}. We emphasize that this functional calculus works in the general framework of quasi-Banach lattices, so given a finite family $\varphi=(f_j)_{i\in J}$ in a quasi-Banach lattice $\LL$, we can define the vector
\[
v_r(\varphi):=\enpar{\sum_{j\in J} \abs{f_j}^r}^{1/r} \in \LL.
\]
The lattice functional calculus allows to complexify any real quasi-Banach lattice.

Most quasi-Banach lattices arise from function quasi-norms. We refer the reader to \cite{AnsorenaBello2022} for the definition and background on this notion. To set the terminology we recall that given a measure space $(\Omega,\Sigma,\mu)$ and a function quasi-norm $\rho\colon L_0^+(\mu)\to [0,\infty]$, we define the quasi-Banach lattice associated with $\rho$ as
\[
\LL_\rho=\enbrace{ f \in L_0(\mu) \colon \rho(\abs{f})<\infty}.
\]

A K\"othe space will be a quasi-Banach lattice constructed from a function quasi-norm with the Fatou property, i.e., a function quasi-norm $\rho$ such that $\rho(\sup_n f_n)=\sup_n \rho(f_n)$ for any non-decreasing sequence $(f_n)_{n=1}^\infty$ in $L_0^+(\mu)$.

Given $0<r<\infty$, the vector $v_r(\varphi)$ associated with a finite family $\varphi=(f_j)_{j\in J}$ in the K\"othe space $\LL_\rho$
takes the easy form
\[
v_r(\varphi)(\omega)=\enpar{\sum_{j\in J} \abs{f_j(\omega)}^r}^{1/r}, \quad \omega\in\Omega.
\]

We will consider vector-valued versions of these spaces. Given a quasi-Banach space $\XX$ endowed with a quasi-norm $\norm{\cdot}_\XX$, we let $L_0(\mu,\XX)$ be the space consisting of all strongly measurable functions from $\Omega$ to $\XX$, where as usual, we identify functions that coincide almost everywhere. We set
\[
\LL_\rho(\XX)=\enbrace{ f \in L_0(\mu,\XX) \colon \rho(\norm{f}_\XX)<\infty}.
\]

We say that $\LL$ is \emph{lattice $r$-convex} if there is a constant $C$ such that $\norm{ v_r(\varphi) } \le C N_r(\varphi)$ for any finite family $\varphi=(f_j)_{j\in J}$ in $\LL$, and we denote by $M^{(p)}(\LL)$ the smallest such a constant $C$ for which this inequality holds. If $\LL$ is a $r$-convex quasi-Banach lattice, then $\LL$ is a $\min\{r,1\}$-convex quasi-Banach space. The converse does not hold; indeed there are quasi-Banach lattices that are not $r$-convex for any $r>0$. Kalton showed in \cite{Kalton1984b} that a quasi-Banach lattice $\LL$ is \emph{lattice $r$-convex} for some $0<r<\infty$ if and only if it is \emph{$L$-convex}, i.e., there is $0<\varepsilon<1$, which we call an $L$-convexity constant for $\LL$, so that
\begin{equation}\label{eqreflat1.1}
\varepsilon \norm{ f } \le \max_{j\in J} \norm{ f_j}
\end{equation}
whenever $f\in\LL$ and the finite family $(f_j)_{j\in J}$ in $\LL$ satisfy
\begin{itemize}
\item $(1-\varepsilon) \abs{J} f\ge \sum_{j\in J} f_j$, and
\item $0\le f_j\le f$ for every $j\in J$.
\end{itemize}
Quantitatively, if $\LL$ is $r$-convex, there is an $L$-convexity constant for $\LL$ that only depends on $r$ and $M^{(p)}(\LL)$; and, the other way around, if $\LL$ is $L$-convex with $L$-convexity constant $\varepsilon$, then there are $r=r(\varepsilon)$ and $C=C(\varepsilon)$ such that $\LL$ is lattice $r$-convex with $M^{(p)}(\LL)\le C$. Note that lattice $1$-convexity is equivalent to $1$-convexity, i.e., to local convexity.

Concavity is, in a sense, the dual property of convexity. We say that a quasi-Banach lattice $\LL$ is \emph{lattice $r$-concave}, $0<r<\infty$, if there is a constant $C$ such that
\begin{equation}\label{eqreflat1.2}
N_r(\varphi)\le C \norm{v_r(\varphi)}
\end{equation}
for any finite family $\varphi$ in $\LL$. In this case we will denote by $M_{(r)}(\LL)$ the smallest constant $C$. We say that $\LL$ is \emph{$L$-concave} if it is $r$-concave for some $r<\infty$.

If we only require \eqref{eqreflat1.1} (resp., \eqref{eqreflat1.2}) to hold in the case when $\varphi=(f_j)_{j\in J}$ is a disjointly supported family, so that
\[
v_r(\varphi)=\abs{\sum_{j\in J} f_j},
\]
we say that $\LL$ satisfies an upper (resp., lower) $r$-estimate. The abovementioned Theorem 2.2 from \cite{Kalton1984b} also gives the following.

\begin{theorem}[see \cite{Kalton1984b}]\label{thm:ConvexInterval}
If an $L$-convex quasi-Banach lattice satisfies an upper $r$-estimate for some $0<r<\infty$, then $\LL$ is lattice $s$-convex for all $0<s<r$. In particular, if $\LL$ is lattice $r$-convex, then it is lattice $s$-convex for all $0<s<r$.
\end{theorem}

There is a lattice-concave counterpart of Theorem~\ref{thm:ConvexInterval} that can be proved using convexifications. Given a quasi-Banach lattice $\LL$ and $0<p<\infty$, we can define its $p$-convexification $\LL^{(p)}$ using the procedure described in \cite{LinTza1979}*{pp.\ 53--54}.

\begin{theorem}\label{thm:ConcaveInterval}
Suppose that a quasi-Banach lattice $\LL$ satisfies a lower $r$-estimate for some $0<r<\infty$. Then $\LL$ is $L$-convex and $s$-concave for all $s>r$.
\end{theorem}

\begin{proof}
An application of \cite{Kalton1984b}*{Theorem 4.1} yields that $\LL$ is $L$-convex. Let $0<p<\infty$ be such that $\LL$ is lattice $p$-convex. Then the quasi-Banach lattice $\LL^{(1/p)}$ is $1$-convex, i.e., it is locally convex and satisfies a lower $r/p$-estimate. Using \cite{LinTza1979}*{Theorem 1.f.7} we deduce that $\LL^{(1/p)}$ is $s$-concave for all $s>r/p$. Therefore $\LL$ is $s$-concave for all $s>r$.
\end{proof}

Notice that the argument we used to prove Theorem~\ref{thm:ConcaveInterval} also works to prove Theorem~\ref{thm:ConvexInterval}. In turn, we can use Theorem~\ref{thm:ConcaveInterval} to obtain the following result of Kalton.

\begin{corollary}[see \cite{Kalton1984b}]
Any $L$-concave quasi-Banach lattice is $L$-convex.
\end{corollary}

Averaging methods are a powerful tool for studying the geometric structure of Banach spaces. Although most of these methods strongly rely on techniques which could fail when local convexity is dropped, it is still worth it to give it a chance to implementing some of these methods within the more general framework of quasi-Banach spaces. To that end, it would be convenient to formulate lattice convexity and lattice concavity in terms of averages.

Given $0<r<\infty$, the \emph{$r$-average} of a finite family $\varphi=(f_j)_{j\in J}$ in a quasi-Banach lattice $\LL$ is defined as
\[
\Ave_{j\in J} (f_j;r)=\enpar{\frac{1}{\abs{J}}\sum_{j\in J} \abs{f_j}^r }^{1/r}=\abs{J}^{-1/r} v_r(\varphi).
\]
The quasi-Banach lattice $\LL$ is $r$-convex with $M^{(r)}(\LL)\le C$ (resp., $r$-concave with $M_{(r)}(\LL)\le C$) if and only if $F\le G$ (resp., $G\le C F$), where
\[
F=\norm{\Ave_{J\in J} (f_j;r)} , \quad G= \Ave_{j\in J} (\norm{f_j};r)
\]
for any finite family $(f_j)_{j\in J}$ in $\LL$.

The lattice functional calculus permits to obtain the following lattice-valued version of Khintchine's inequalities.

\begin{theorem}\label{lem:KintchineLattice}
There are constants $T_r$ and $C_r$ such that
\begin{equation*}
\frac{1}{C_r} \Ave_{\varepsilon_j=\pm 1} \left(\sum_{j\in J} \varepsilon_j\, f_j;r\right)
\le \enpar{\sum_{j\in J} \abs{f_j}^2 }^{1/2}
\le T_r \Ave_{\varepsilon_j=\pm 1} \left(\sum_{j\in J} \varepsilon_j\, f_j;r\right)
\end{equation*}
for every quasi-Banach lattice $\LL$ and every finite
family $(f_j)_{J\in J}$ in $\LL$.
\end{theorem}

\begin{proof}
Just apply the lattice functional calculus to the functions $f$, $g\colon\RR^n \to \RR$ given for each $n\in\NN$ by
\[
f((x_i)_{i=1}^n)=\left(\sum_{i=1}^n |x_i|^2\right)^{1/2}, \quad g((x_i)_{i=1}^n)= \left(\Ave_{\varepsilon_i=\pm 1} \left|\sum_{i=1}^n \varepsilon_i\, x_i\right|^{r}\right)^{1/r},
\]
and use Khintchine's inequalities \cite{Kinchine1923}.
\end{proof}

Kalton \cite{Kalton1984b} observed that $L$-convex lattices behave similarly to Banach lattices in many ways. Lemma~\ref{lem:MaureyQB} below, which generalizes to quasi-Banach lattices \cite{Maurey1974}*{Lemme 5 and Lemme 6} (see also \cite{LinTza1979}*{Theorem 1.d.6}), is in this spirit.

Give a finite family $\varphi=(f_j)_{j\in J}$ in a quasi-Banach space $\XX$ and $0<r<\infty$ we set
\[
A_r(\varphi)= \Ave_{\varepsilon_j=\pm 1}\enpar{\norm{ \sum_{J\in J} \varepsilon_j\, f_j};r}.
\]
In the case when $\XX$ is quasi-Banach lattice we also set
\[
L_2(\varphi)=\norm{\enpar{\sum_{j\in J} \abs{f_j}^2 }^{1/2}}.
\]

\begin{lemma}\label{lem:MaureyQB}
Let $\LL$ be a quasi-Banach lattice.
\begin{enumerate}[label=(\roman*), leftmargin=*, widest=ii]
\item\label{lem:MaureyQB:Convex} If $\LL$ is $L$-convex, then for every $0<r<\infty$ there is constant $C$ such that $L_2(\varphi) \le C A_r(\varphi)$ for every finite family $\varphi$ in $\LL$.

\item\label{lem:MaureyQB:Concave} If $\LL$ is $L$-concave, then for every $0<r<\infty$ there is a constant $C$ such that $L_2(\varphi) \le C A_r(\varphi)$ and $ A_r(\varphi) \le L_2(\varphi)$ for every finite family $\varphi$ in $\LL$.
\end{enumerate}
\end{lemma}

\begin{proof}
Combining Theorem~\ref{lem:KintchineLattice}, Theorem~\ref{thm:ConcaveInterval} and Theorem~\ref{thm:ConvexInterval} yields $0<s<r$ and a constant $C_0$ such that $L_2(\varphi) \le C_0 A_s(\varphi)$ for all finite families $\varphi$ in $\LL$. Also, if $\LL$ has some nontrivial concavity there are $t>r$ and a constant $C_1$ such that $A_t(\varphi)\le C_1 L_2(\varphi)$ for all finite families $\varphi$ in $\LL$. Since the map $r\mapsto A_r(\varphi)$ is non-decreasing, we are done.
\end{proof}

Let $(\varepsilon_n)_{n=1}^\infty$ be a sequence of Rademacher random variables in a probability space $(\Omega,P)$. Given a quasi-Banah space $\XX$ and $0<r<\infty$ we define $\Rad(\XX)$ as the closed linear span in $L_r(P,\XX)$ of the set
\[
\{x\, \varepsilon_n\colon x\in\XX, \, n\in\NN\}.
\]
In light of Kahane-Kalton inequalities, the average $A_r(\varphi)$ associated with a finite family $\varphi$ in a quasi-Banach space $\XX$ does not essentially depend on the index $r\in(0,\infty)$, so the same goes for $\Rad(\XX)$. Lemma~\ref{lem:MaureyQB} provides a simple proof of Kahane-Kalton inequalities in the case when the target space is an $L$-concave quasi-Banach lattice, which permits to identify the space $\Rad(\XX)$. Namely, it readily gives the following result.

\begin{proposition}\label{prop:RadL2}
Let $\LL$ be an $L$-concave K\"othe space. Let $(\varepsilon_n)_{n=1}^\infty$ be a sequence of Rademacher random variables.
Then the mapping
\[
(x_n)_{n=1}^\infty \mapsto \sum_{n=1}^\infty x_n \varepsilon_n
\]
defines an isomorphism from $\LL(\ell_2)$ onto $\Rad(\LL)$.
\end{proposition}

To frame Proposition~\ref{prop:RadL2}, observe that if $\XX$ is a Banach space, the Bochner integral defines a bounded, linear, one-to-one map
\[
\Rad(\XX)\to c_0(\XX), \quad f\mapsto \enpar{\int_\Omega f \, \varepsilon_n\, dP}_{n=1}^\infty,
\]
so that we can think of the members of $\Rad(\XX)$ as $\XX$-valued sequences. However, if $\XX$ is non-locally convex, the Bochner integral is no longer available (see \cite{AA2013}), so such representation is uncertain.

We conclude this section by recording two theorems from \cite{Kalton1984b} that will be essential for our study. The first one tells us that $L$-convexity is inherited not only by sublattices but also by subspaces.

\begin{lemma}[see \cite{Kalton1984b}*{Theorem~4.2}]\label{thm:LCNatural}
Let $\LL_0$ and $\LL$ be quasi-Banach lattices. Assume that $\LL$ is $L$-convex and that $\LL_0$ is linearly isomorphic to a subspace of $\LL$. Then $\LL_0$ is $L$-convex.
\end{lemma}

The second one states that operators between $L$-convex spaces always admit Hilbert-valued extensions.

\begin{theorem}[see \cite{Kalton1984b}*{Theorem~3.3}]\label{thm:KaltonVectorizes}
Let $\LL$ be an $L$-convex quasi-Banach lattice with $L$-convexity constant $\varepsilon$. There is a constant $C$ depending only on $\varepsilon$ such that for any quasi-Banach lattice $\XX$ and any bounded linear operator $T\colon\XX\to\LL$,
\[
\norm{ \enpar{\sum_{j\in J} \abs{T(f_j)}^2}^{1/2}} \le C \norm{T} \norm{ \enpar{\sum_{j\in J} \abs{f_j}^2}^{1/2}}
\]
for any finite family $(f_j)_{j\in J}$ in $\XX$.
\end{theorem}

\section{Vector-valued Hardy spaces in various positions}\label{sect:VVHE}\noindent
In this section we will introduce various classes of vector-valued Hardy spaces $H_p$ for $0<p\le 1$ and we will establish some identifications amongst them.

We will denote by $\DD$ the open unit disc of the complex field, and by $\TT$ the algebraic torus. As is customary, given $0<\tau<\infty$ we regard functions on $\TT$ as $\tau$-periodic functions on $\RR$, via the topological group epimorphism
\[
t\mapsto e^{2\pi i t/\tau} , \quad t \in\RR,
\]
whose kernel is $\tau \ZZ$. Any function defined on a left-closed, right-open interval has a unique periodic extension, so it can be seen as a function defined on $\TT$. This way we can identify the algebraic torus with any real interval, the most common ones being $[0,2\pi)$, $[-\pi,\pi)$, $[-1,1)$, $[0,1)$, and $[-1/2,1/2)$. In what follows we will work with $I:=[-1/2,1/2)$, which has length one and it is centered at the origin. The target space of our functions will be a Banach space $\XX$ over the real or complex field $\FF$ endowed with a norm $\norm{\cdot}_\XX$.

\subsection{Banach-valued maximal Hardy spaces.}
Let $\Dt$ be the set of all dyadic intervals contained in $I$. For each $n\in\NN_0$, let $\Dt_n$ consist of all $J\in\Dt$ with $\abs{J}= 2^{-n}$, and let $\Delta_n$ be the finite $\sigma$-algebra generated by $\Dt_n$. Let $\Mt(I,\XX)$ be the vector space of all $\XX$-valued dyadic martingales over $I$, i.e., the set of all sequences $F=(f_n)_{n=0}^\infty$ such that for each $n\in\NN_0$ the function $f_n\colon I \to\XX$ is $\Delta_n$-measurable, with
\[
\EE(f_{n+1},\Delta_{n})=f_n.
\]
We define the \emph{maximal function} of the martingale $F$ as
\[
F^*\colon I \to [0,\infty], \quad F^*=\sup_{n\in\NN_0} \norm{f_n}_\XX.
\]
We then put for each $F\in \Mt(I,\XX)$,
\[
\norm{F}_{H_p^{\delta,m}(I,\XX)}=\norm{F^*}_{L_p(I)}\in[0,\infty],
\]
and define the $\XX$-valued \emph{maximal dyadic Hardy space of index $p\in(0,1]$} as the quasi-Banach space
\[
H_p^{\delta,m}(I,\XX)= \enbrace{F\in \Mt(I,\XX) \colon \norm{F}_{H_p^{\delta,m}(I,\XX)}<\infty}.
\]

In some situations it is convenient to deal with Hardy spaces of measurable functions. On this matter we note that the martingale transform
\begin{equation}\label{eq:martingaletransform}
\delta\colon L_1(I,\XX) \to \Mt(I,\XX), \quad f \mapsto \delta(f):=(\EE(f,\Delta_n))_{n=0}^\infty,
\end{equation}
defines a natural embedding of $L_1(I,\XX)$ into $\Mt(I,\XX)$. We will thus regard $L_1(I,\XX)$ as a subspace of $\Mt(I,\XX)$, and will define $\HL_p^{\delta,m}(I,\XX)$ as the closure of $H_p^{\delta,m}(I,\XX)\cap L_1(I,\XX)$ in $H_p^{\delta,m}(I,\XX)$. The overlapping of $L_1(I,\XX)$ and this family of Hardy spaces depends on whether $p=1$ or $0<p<1$.

Given $f\in L_1(I,\XX)$ and a measurable set $A\subseteq I$ we put
\[
\Ave(f,A)= \frac{1}{\abs{A}} \int_A f.
\]
In turn, given a dyadic martingale $F=(f_n)_{n=0}^\infty$ we define
\[
\Ave(F,J)= \frac{1}{\abs{J}} \int_J f_n, \quad J\in \bigcup_{j=0}^n\Dt_j.
\]
Note that $\Ave(f,J)=\Ave(\delta(f),J)$ for all $f\in L_1(I,\XX)$ and $J\in\Dt$. Note also that
\begin{equation*}
F^*(t) = \sup_{t\in J \in \Dt} \norm{\Ave(F,J)}_\XX \quad F\in \Mt(I,\XX), \, t\in I.
\end{equation*}
Therefore, given $f\in L_1(I,\XX)$, by Lebesgue's differentation theorem (see, e.g., \cite{BenLin2000}*{Proposition 5.3}), $\norm{f}_\XX \le (\delta(f))^*$ a.e., so that
\[
\norm{f}_{L_1(I,\XX)} \le \norm{\delta(f)}_{H_1^{\delta,m}(I,\XX)}.
\]
We infer that $\HL_1^{\delta,m}(I,\XX)=H_1^{\delta,m}(I,\XX)\cap L_1(I,\XX)$, whence $\HL_1^{\delta,m}(I,\XX)=H_1^{\delta,m}(I,\XX)$ if and only if $H_1^{\delta,m}(I,\XX) \subseteq L_1(I,\XX)$.

Suppose now that $0<p<1$. Since $L_{1,\infty}(I)\subseteq L_p(I)$ and the Hardy-Littlewood maximal function $M_{HL}$ is of weak type $(1,1)$, there is a constant $C$ such that, for all $f\in L_1(I,\XX)$,
\begin{multline*}
\norm{\delta(f)}_{H_p^{\delta,m}(I,\XX)}
\le\frac{1}{1-p} \norm{\delta(f)^*}_{L_{1,\infty}(I)} \\
\le\frac{1}{1-p}\norm{M_{HL}\enpar{\norm{f}_\XX}}_{L_{1,\infty}(I)}
\le \frac{C}{1-p}\norm{f}_{L_1(I,\XX)}.
\end{multline*}
Consequently, $L_1(I,\XX)\subseteq H_p^{\delta,m}(I,\XX)$, whence $\HL_p^{\delta,m}(I,\XX)=H_p^{\delta,m}(I,\XX)$ if and only if
$L_1(I,\XX)$ is a dense subspace of $H_p^{\delta,m}(I,\XX)$.

Taking into account that the embedding $H_p^{\delta,m}(I,\XX)\subseteq \ell_\infty(L_p(I,\XX))$ is a contraction, a standard $\varepsilon/3$-argument gives that a function $F=(f_n)_{n=0}^\infty$ belongs to $\HL_p^{\delta,m}(I,\XX)$ if and only if $\lim_n \delta(f_n)=F$ in the norm-topology of $H_p^{\delta,m}(I,\XX)$. So, regardless of whether $p<1$ or $p=1$, there is a `limit-value' bounded linear map
\[
\HL_p^{\delta,m}(I,\XX) \to L_p(I,\XX), \quad (f_n)_{n=0}^\infty \mapsto \lim_n f_n,
\]
where the limit is taken in the sense of $L_p(I,\XX)$.

As far as the $L_q$-spaces for $q>1$ is concerned, we note that since $M_{HL}$ is of type $(r,r)$ for all $r\in(1,\infty)$, $\delta$ is a bounded operator from $L_q(I,\XX)$ to $H_1^{\delta,m}(I,\XX)$. Summing up, if we identify martingales with their limit functions we obtain
\begin{equation}\label{eq:dyadichain}
L_q(I,\XX) \subseteq \HL_1^{\delta,m}(I,\XX) \subseteq L_1(I,\XX) \subseteq \HL_p^{\delta,m}(I,\XX) \subseteq L_p(I,\XX)
\end{equation}
continuously for $0<p<1<q\le \infty$.

A Banach space $\XX$ is said to have the RNP if the statement of Radon-Nikod\'ym holds for $\XX$-valued measures. This property can be characterised in terms of martingales. To be precise we will use that, as Bourgain and Rosenthal claimed in \cite{BR1980}*{Proposition in page 69}, $\XX$ has the RNP if and only for every uniformly bounded dyadic martingale, i.e., any dyadic martingale $F=(f_n)_{n=0}^\infty$ with
\[
\Vert F^*\Vert_\infty=\sup_n \Vert f_n\Vert_\infty<\infty
\]
there is $f\in L_\infty(I,\XX)$ such that $\delta(f)=F$. Indeed, the classical construction of a uniformly bounded martingale $(f_n)_{n=0}^\infty$ over $[0,1)$ with
\begin{equation}\label{eq:tree}
\inf\enbrace{ \norm{f_n(t) - f_{n-1}(t)}_\XX \colon t\in[0,1), \, n\in\NN} >0
\end{equation}
from a non-dentable set of the unit ball of $\XX$ (see, e.g., \cite{BenLin2000}*{Proof of Theorem 5.8}) can easily adapted to yield a subsequence of a dyadic martingale. In contrast, there are Banach spaces $\XX$ with no $\XX$-valued uniformly bounded dyadic martingale satisfying \eqref{eq:tree}, and still do not have the RNP (see \cite{BR1980}*{Theorem 2.4}).

\begin{theorem}\label{thm:density}
Given $0<p\le 1$ and a Banach space $\XX$, $\HL_p^{\delta,m}(I,\XX)=H_p^{\delta,m}(I,\XX)$ if and only if $\XX$ has the RNP.
\end{theorem}

We will base the proof of Theorem~\ref{thm:density} on a version for vector-valued martingales of Calder\'on-Zygmund's decomposition of functions into a `good' part and a `bad' part, a technique that goes back to \cite{CZ1952}.

Let $\Mt_0(I,\XX)$ be the subspace of $\Mt(I,\XX)$ consisting of all $\XX$-valued dyadic martingales $(f_n)_{n=0}^\infty$ with $f_0=0$. If $F=(f_n)_{n=0}^\infty\in\Mt_0(I,\XX)$ then the unit interval $I$ does not belong to
\[
\Et_\lambda=\enbrace{ J \in \Dt \colon \norm{\Ave(F,J)}_\XX>\lambda}
\]
for any $\lambda>0$. We have
\[
U_\lambda:=\enbrace{t\in[0,1) \colon F^*(t)>\lambda}=\bigcup_{J\in\Et_\lambda} J.
\]
The family $(U_\lambda)_{\lambda>0}$ decreases as $s$ increases, and
\begin{equation}\label{eq:integral}
\norm{F}_{H_p^{\delta,m}(I,\XX)} = \enpar{ \int_0^\infty \lambda^p \abs{U_\lambda}\, d\lambda}^{1/p}.
\end{equation}

Let $\Dt_\lambda$ be the set of all maximal intervals of $\Et_\lambda$. Note that the intervals of $\Dt_\lambda$ are pairwise disjoint, and their union is $U_\lambda$.

For each $J\in\Dt$, let $D(J)$ denote the dyadic cube containing $J$ of length twice the length of $J$. Since $I\notin \Et_\lambda$, $D(J)\in \Dt$ for all $J\in \Dt_\lambda$. The maximality of the intervals of $\Dt_\lambda$ implies that $D(J)$ is not contained in $U_\lambda$ for any $J\in \Dt_\lambda$. Hence,
\begin{equation}\label{eq:AverageDoubling}
\norm{\Ave(F,D(J))} \le \lambda, \quad J\in \Dt_\lambda.
\end{equation}

Let $\Bt_\lambda\subseteq \Dt_\lambda$ be set of all intervals $J$ such that $D(J)$ is maximal within the family of sets $(D(J))_{J\in \Dt_\lambda}$. As before, $(D(J))_{J\in \Bt_\lambda}$ is a family of pairwise disjoint sets, and
\[
V_\lambda:=\bigcup_{J\in \Dt_\lambda} D(J) = \bigcup_{J\in \Bt_\lambda} D(J).
\]
We have
\begin{equation} \label{eq:VsUs}
\abs{V_\lambda} \le \sum_{J\in \Dt_\lambda} \abs{D(J)} \le 2 \sum_{J\in \Dt_\lambda} \abs{J}=2\abs{U_\lambda}.
\end{equation}
For each $n\in\NN$, set
\[
\Bt_{\lambda,n}=\Bt_\lambda\cap \enpar{\bigcup_{j=0}^n \Dt_j},
\]
$V_{\lambda,n}=\cup_{J\in\Bt_{\lambda,n}} D(J)$, and
\[
\Gt_{\lambda,n}=\enbrace{J\in\Dt_n \colon J \cap V_{\lambda,n}=\emptyset}.
\]
Let $\Ft_{\lambda,n}$ be the $\sigma$-algebra generated by
\[
\enbrace{D(J) \colon J\in \Bt_{\lambda,n}} \cup \Gt_{\lambda,n}.
\]
Set also $\Ft_{\lambda,0}=\{\emptyset, I\}$. The filtration $(\Ft_{\lambda,n})_{n=0}^\infty$ is coarser than $(\Delta_n)_{n=0}^\infty$. Then, if for $n\in\NN_0$ we let $\EE(f_n,\Ft_n)=g_{n,\lambda}$,
\begin{equation}\label{eq:goodbounded}
\norm{g_{n,\lambda}}_\XX \le \norm{f_n}_\XX, \quad n\in\NN_0.
\end{equation}
By the maximality of the members of $\Bt_{\lambda,n}$,
\[
G_\lambda:=\enpar{g_{n,\lambda}}_{n=0}^\infty\in \Mt_0(I,\XX).
\]
Set $B_\lambda:=F-G_\lambda$.

In this setting we say that $(G_\lambda,B_\lambda)$ is the \emph{Calder\'on-Zygmund decomposition of $F$ at threshold $\lambda$}. We will refer to
\[
\Ct_\lambda:=\enbrace{D(J) \colon J \in \Bt_\lambda}
\]
as the \emph{collection of Calder\'on-Zygmund intervals of $F$ at threshold $\lambda$}, and we will say that $W_\lambda:=I \setminus V_\lambda$ is the \emph{Calder\'on-Zygmund measurable set of $F$ at threshold $\lambda$}.

We denote by $\Sigma_\lambda$ the $\sigma$-algebra over $I$ consisting of all measurable sets that either contain or are disjoint with each set in $\Ct_\lambda$. We call $\Sigma_\lambda$ the \emph{Calder\'on-Zygmund $\sigma$-algebra of $F$ at threshold $\lambda$}, and it determines the partition $\Ct_\lambda \cup\{W_\lambda\}$ of $I$.

To compare decompositions at different thresholds, we pick $0<\lambda\le \eta<\infty$, $H\in \Ct_\eta$, and $A\in\Sigma_\lambda$ with $A\cap D(H)\not=\emptyset$. Since $U_\eta\subseteq U_\lambda$, there is $K\in\Dt_\lambda$ with $H\subseteq K$. In turn, there is $J\in \Ct_\lambda$ such that $D(K)\subseteq D(J)$. Since $D(H)\subseteq D(K)$, it follows that $D(H)\subseteq D(J)$, whence $D(J)\cap A\not=\emptyset$. Therefore $D(J)\subseteq A$ and so $D(H)\subseteq A$. This proves that
\begin{enumerate}[label=(CZ.\arabic*), leftmargin=*]
\item\label{it:GBSA} the family $(\Sigma_\lambda)_{\lambda>0}$ of Calder\'on-Zygmund $\sigma$-algebras of $F\in \Mt_0(I,\XX)$ increases as $\lambda$ increases.
\end{enumerate}

Combining \eqref{eq:integral}, \eqref{eq:VsUs} and \ref{it:GBSA} gives that
\begin{enumerate}[label=(CZ.\arabic*), leftmargin=*,resume]
\item\label{it:VGB} the family $(W_\lambda)_{\lambda>0}$ of Calder\'on-Zygmund measurable sets of $F\in \Mt_0(I,\XX)$ increases as $\lambda$ increases, and for $0<p\le 1$,
\[
\norm{F}_{H_p^{\delta,m}(I,\XX)} \le \enpar{ \int_0^\infty \lambda^p \abs{I\setminus W_\lambda}\, d\lambda}^{1/p} \le 2^{1/p} \norm{F}_{H_p^{\delta,m}(I,\XX)}.
\]
In particular,
\[
\lim_{\lambda\to \infty} \abs{I \setminus W_\lambda}=0
\]
provided that $F\in H_p^{\delta,m}(I,\XX)$.
\end{enumerate}

We next show that both components of the Calder\'on-Zygmund decompositon of $F\in \Mt_0(I,\XX)$ can be controlled. If $J\in \Et_\lambda$, then $J\subseteq J_1$ for some $J_1\in\Dt_\lambda$. In turn, $D(J_1)\subseteq D(J_2)$ for some $J_2\in \Bt_\lambda$. Clearly, $\abs{J} \le \abs{J_2}$ and $J\subseteq D(J_2)$. Therefore, $\Et_\lambda\cap \Gt_{\lambda,n}=\emptyset$ for all $n\in\NN$. By \eqref{eq:AverageDoubling},
\begin{enumerate}[label=(CZ.\arabic*), leftmargin=*,resume]
\item\label{it:G} The good component $G_\lambda$ at threshold $\lambda\in(0,\infty)$ of any $F\in\Mt_0(I,\XX)$ satisfies $\norm{G_\lambda^*}_\infty \le \lambda$.
\end{enumerate}

Let $t\in I \setminus V_\lambda$ and $n\in\NN$. Then, $t\in J$ for some $J\in \Gt_{\lambda,n}$. Therefore $G_\lambda(t)=\Ave(f_n,J)=f_n(t)$. Consequently, by
\eqref{eq:VsUs} and \eqref{eq:goodbounded},
\begin{enumerate}[label=(CZ.\arabic*), leftmargin=*,resume]
\item\label{it:B} The bad component $B_\lambda$ and the Calder\'on-Zygmund measurable set $W_\lambda$ at threshold $\lambda\in(0,\infty)$ of any $F\in\Mt(I,\XX)$ satisfy
\[
B_\lambda^* \le 2 F^* \chi_{I\setminus W_\lambda}.
\]
\end{enumerate}

If $F$ is a function, so are $G_\lambda$ and $B_\lambda$. Precisely, if $f\in L_1(I,\XX)$, the good part of $\delta(f)$ at threshold $\lambda$ is $\delta(g_\lambda)$, where
\[
g_\lambda=\EE(f,\Sigma_\lambda)= f \chi_{W_\lambda}+\sum_{J \in \Ct_\lambda} \Ave(f,J) \chi_{J}.
\]
We call $(g_\lambda,f-g_\lambda)$ the Calder\'on-Zygmund decomposition of $f$ at threshold $\lambda$.

\begin{proof}[Proof of Theorem~\ref{thm:density}]
Suppose that $\XX$ has the RNP, and pick $F=(f_n)_{n=0}^\infty\in H_p^{\delta,m}(I,\XX)$. Let $(G_\lambda,B_\lambda)$ be the Calder\'on-Zygmund decomposition of $F-\delta(f_0)$ at threshold $\lambda\in(0,\infty)$. By \ref{it:VGB}, \ref{it:B} and Lebesgue's dominated convergence theorem, $\lim_{\lambda\to\infty} G_\lambda=F-\delta(f_0)$ in the norm-topology of $H_p^{\delta,m}(I,\XX)$. By \ref{it:G}, $G_\lambda\in \delta(L_\infty(I,\XX))$ for all $\lambda\in(0,\infty)$.

Suppose now that $\HL_p^{\delta,m}(I,\XX)=H_p^{\delta,m}(I,\XX)$, and pick a uniformly bounded martingale $F=(f_n)_{n=0}^\infty\in \Mt(I,\XX)$. There is a function $f\colon I \to \XX$ such that $\lim_n f_n=f$ in the norm-topology of $L_p(I,\XX)$. By Lebesgue's dominated convergence theorem, $f\in L_\infty(I,\XX)$ and $\lim_n f_n=f$ in the norm-topology of $L_1(I,\XX)$. Hence, $F=\delta(f)$.
\end{proof}

To introduce Hardy spaces defined via the Poisson maximal function we consider the vector space $\Dt'(\TT,\XX)$ of all continuous linear maps from the topological vector space $\Ct^{(\infty)}(\TT,\FF)$ into $\XX$.

Let $P_r\colon\RR\to\RR$ be the Poisson kernel at a point $r\in[0,1)$, defined by $P_r(t)=\mathbf{P}(r e^{2\pi i t})$, where
\[
\mathbf{P}(z)=\Re \enpar{\frac{1+z}{1-z}}.
\]
Given $f\in \Dt'(\TT,\XX)$, put
\[
M(f)\colon \RR \to [0,\infty), \quad t \mapsto \sup_{0\le r<1} \norm{P_r*f(t)}_\XX.
\]
The $\XX$-valued \emph{maximal Hardy space of index $p\in(0,1]$} will be the quasi-Banach space of all $f\in \Dt'(\TT,\XX)$ such that
$M(f)\in L_p(\TT)$, i.e.,
\[
H_p^m(\TT,\XX)=\enbrace{f\in\Dt'(\TT,\XX) \colon \norm{f}_{H_p^m(\TT,\XX)}<\infty},
\]
where
\[
\norm{f}_{H_p^m(\TT,\XX)} =\norm{M(f)}_{L_p(I)}.
\]
We think of functions in $L_1(\TT,\XX)$ as distributions in $\Dt'(\TT,\XX)$, and define $\HL_p^m(\TT,\XX)$ as the closure of
\[
H_p^m(\TT,\XX)\cap L_1(\TT,\XX)
\]
in $H_p^m(\TT,\XX)$. Blasco and Garc\'{\i}a-Cuerva \cite{BGC1987} proved results that run parallel to the ones for dyadic spaces recorded above, and which we list next for reference.

\begin{theorem}[\cite{BGC1987}*{Theorem 1.1}]\label{thm:BGC}
Let $0<p\le 1$ and $\XX$ be a Banach space.
\begin{enumerate}[label=(\roman*), leftmargin=*,widest=iii]
\item If $p<1$, $L_1(\TT,\XX)\subseteq H_p^m(\TT,\XX)$.
\item $H_1^m(\TT,\XX)\cap L_1(\TT,\XX)=\HL_1^m(\TT,\XX)$.
\item $\HL_p^m(\TT,\XX)=H_p^m(\TT,\XX)$ if and only if $\XX$ has the RNP property.
\end{enumerate}
\end{theorem}

\subsection{Banach-valued atomic Hardy spaces.}
An \emph{$\XX$-valued dyadic $p$-atom} associated with an interval $J\in \Dt$ will be a strongly measurable function $a\colon I \to \XX$ such that
\begin{enumerate}[label=(A.\arabic*)]
\item $\supp(a)\subseteq J$,
\item $\norm{a}_{L_\infty(I,\XX) }\le \abs{J}^{-1/p}$, and
\item\label{it:atom:cancel} $ \int_I a=0$.
\end{enumerate}

Let $\At^\delta_p(I,\XX)$ be the set of all $a\colon I\to \XX$ that are either constant functions attaining its single value in the unit ball of $B_\XX$ or are dyadic $p$-atoms. Since
\[
\At^\delta_p(I,\XX) \subseteq B_{L_p(I,\XX)},
\]
given a countable family $(a_j)_{j\in \Jt}$ in $\At^\delta_p(I,\XX)$, and a scalar-value family $(\lambda_j)_{j\in \Jt}$ in $\ell_p(\Jt)$ we can safely define
\[
\sum_{j\in\Jt} \lambda_j \, a_j \in L_p(I,\XX).
\]
The $\XX$-valued \emph{atomic dyadic Hardy space} $H_p^{\delta,a}(I,\XX)$ will be the quasi-Banach space of all functions $f\in L_p(I,\XX)$ that can be expanded in this way endowed with the $p$-norm
\[
f\mapsto \inf\enbrace{ \enpar{\sum_{j\in \Jt} \abs{\lambda_j}^p}^{1/p} \colon f=\sum_{j\in\Jt} \lambda_j \, a_j }.
\]
In other words, $H_p^{\delta,a}(I,\XX)$ is the $p$-Banach space constructed from $\At^\delta_p(I,\XX)$ by means of the $p$-convexification method (see, e.g., \cite{AACD2018}*{Section~2.2}).

$p$-atoms associated with an arbitrary interval $J\subseteq I$ are defined in a slightly different manner. For each $j\in\NN_0$ we denote by $p_j$ the monomial
\[
t\mapsto p_j(t)=t^j.
\]
Instead of \ref{it:atom:cancel} we impose to the atom $a$ to have null moments up to orden $1/p-1$, that is,
\begin{equation}\label{it:moments}
\int_I a, p_j = \int_I a(t)\, t^j \, dt=0, \quad 0\le j \le \floor{\frac{1}{p}}-1.
\end{equation}
Let $\At_p(I,\XX)$ be the set of all functions from $I$ to $\XX$ that are either $p$-atoms or members of the set
\[
\enbrace{x \, p_j \colon x\in B_\XX, \, 0\le j \le \floor{\frac{1}{p}}-1}.
\]
The atomic Hardy space $H_p^{a}(I,\XX)$ is defined similarly to the atomic dyadic Hardy space, replacing $\At^\delta_p(I,\XX)$ with $\At_p(I,\XX)$.

Coifman and Weiss \cite{CW1977} proved that maximal dyadic spaces and atomic dyadic spaces coincide within the framework of scalar-valued Hardy spaces of index $1$. As we show next, their arguments can be translated to Banach-valued Hardy spaces of index $0<p<1$ despite the fact that, on one hand, local convexity is an essential ingredient which constrains the estimates that can be carried out in those spaces, and, on the other hand, the achievement of vector-valued results may critically depend of the geometry of the target space.

\begin{theorem}\label{thm:Anso}
Let $0<p\le 1$ and $\XX$ be a Banach space. Then $\HL_p^{\delta,m}(I,\XX)=H_p^{\delta,a}(I,\XX)$.
\end{theorem}

\begin{proof}
Let $a$ be an $\XX$-valued dyadic $p$-atom associated with an interval $J\in\Dt$. Let $x\in I$, and let $H\in\Dt$ with $x\in H$. If $H\cap J=\emptyset$ or $J\subseteq H$, then $\int_H a=0$. In turn, if $H\subseteq J$, then $x\in J$. Consequently, $\supp(\delta(a)^*)\subseteq J$, whence
\[
\delta(a)^* \le \Vert a\Vert_{L_\infty(I,\XX)} \chi_J \le \abs{J}^{-1/p} \chi_J.
\]
Therefore, $\norm{\delta(a)^*}_{L_p(I)} \le 1$. We infer that, via the martingale transform $\delta$, $H_p^{\delta,a}(I,\XX) \subseteq H_p^{\delta,m}(I,\XX)$, and the embedding is a contraction.

Assume that $f\in H_p^{\delta,m}(I,\XX) \cap L_1(I,\XX)$. If $f$ is constant, then $f=\lambda a$, where $a\in B_\XX$ and $\lambda=\norm{f}_{H_p^{\delta,m}(I,\XX)}$. Thus, we can assume that $\int_I f=0$. For each $\lambda>0$, let $(g_\lambda,b_\lambda)$ be the Calder\'on-Zygmund decomposition, $\Ct_\lambda$ be the collection of Calder\'on-Zygmund intervals, $W_\lambda$ be the Calder\'on-Zygmund measurable set, and $\Sigma_\lambda$ be the Calder\'on-Zygmund $\sigma$-algebra of $f$ at threshold $\lambda$.

Let $0<\lambda<\eta<\infty$. By \ref{it:GBSA},
\[
\EE(f_{\lambda,\eta},\Sigma_\lambda)=0, \mbox{ where }
f_{\lambda,\eta}:=g_\eta-g_\lambda=b_\lambda-b_\eta.
\]
This means that $f_{\lambda,\eta}|_{W_\lambda}=0$ and $\Ave(f_{\lambda,\eta},J)=0$ for all $J\in \Ct_\lambda$. By \ref{it:G}, $\norm{f_{\lambda,\eta}}_{L_\infty(I,\XX)} \le \lambda+\eta$. It follows that there are families $(\lambda_{J}^{\lambda,\eta})_{J\in \Ct_\lambda}$ and $(a_{J}^{\lambda,\eta})_{J\in \Ct_\lambda}$ such that
\begin{itemize}
\item $f_{\lambda,\eta}=\sum_{J\in \Ct_\lambda} \lambda_{J}^{\lambda,\eta} \, a_{J}^{\lambda,\eta}$ and, for each $J\in\Ct_\lambda$,
\item $a_{J}^{\lambda,\eta}$ is a dyadic $p$-atom associated with the interval $J$, and
\item $0 \le \lambda_{J}^{\lambda,\eta} \le (\lambda+\eta) \abs{J}^{1/p}$.
\end{itemize}

Set $\lambda_{J,k}= \lambda_{J}^{2^k,2^{k+1}}$ for all $k\in\ZZ$ and $J\in\Ct_{2^k}$. By \ref{it:VGB},
\[
\sum_{\substack{ k\in\ZZ\\ J\in \Et_{2^k} }} \lambda_{J,k}^p
\le 3^p \sum_{k\in\ZZ} 2^{kp} \abs{I \setminus W_{2^k}}
\le 2 \cdot 6^p \norm{f}_{H_p^{\delta,m}([0,1),\XX)}.
\]
Hence, if $a_{J,k}= a_{J}^{2^k,2^{k+1}}$, the expansions
\[
h:=\sum_{\substack{ k\in\ZZ\\ J\in \Et_{2^k} }} \lambda_{J,k} \, a_{J,k}, \quad h_{m,n}:=\sum_{\substack{ k\in\ZZ, -m \le k \le n\\ J\in \Et_{2^k} }}
\lambda_{J,k} \, a_{J,k},\quad m,\, n\in\NN,
\]
define functions in $H_p^{\delta,a}(I,\XX)$ with
\[
\norm{h}_{H_p^{\delta,a}(I,\XX)}\le 2^{1/p} \, 6 \norm{f}_{H_p^{\delta,m}(I,\XX)} \mbox{ and }
\lim_{\substack{m\to\infty \\ n \to \infty}} h_{m,n}=h
\]
in norm-topology of $H_p^{\delta,a}(I,\XX)$. We have $h_{m,n}=f-b_{2^{n+1}}-g_{2^{-m}}$. By \ref{eq:AverageDoubling}, $\lim_m g_{2^{-m}}=0$ in $L_\infty(I,\XX)$. By \ref{it:VGB}, \ref{it:B}, and Lebesgue's dominated convergence theorem, $\lim_n b_{2^{n}}=0$ in $L_1(I,\XX)$. Summing up, $h=f$.
\end{proof}

$\XX$-valued atomic Hardy spaces and $\XX$-valued maximal Hardy spaces also coincide for an arbitrary Banach space $\XX$. Before showing this we note that an application of the Gram-Schmidt process to the sequence $(p_j)_{j=0}^\infty$ yields operators
\[
Q_\sigma\colon L_1(I,\XX)\to L_\infty(I,\XX), \quad \sigma\in\NN_0,
\]
with
\[
\int_I Q_\sigma (f) \, p_j=0, \quad j,\sigma\in\NN_0, \, j\le \sigma,\, f\in L_1(I,\XX),
\]
and polynomials $\psi_j$, $j\in\NN_0$, of degree $j$ such that
\begin{equation}\label{eq:GSExpansion}
f=\sum_{j=0}^\sigma \enpar{\int_I f\, \psi_j } p_j + Q_\sigma(f),\quad \sigma\in\NN_0, \, f\in L_1(I,\XX).
\end{equation}
Each operator $f\mapsto \int_I f\, \psi_j$ is bounded from $L_1(I,\XX)$ into $\XX$. Besides, if $\LL$ is a K\"othe space over $I$ containing $L_1(I)$, $Q_\sigma$ is and endomorphism of $\LL(\XX)$ for each $\sigma\in\NN_0$. We call $(Q_\sigma)_{\sigma=0}^\infty$ the Gram-Schmidt projections relative to $(p_j)_{j=0}^\infty$.

\begin{theorem}\label{thm:BGC+Mol}
Let $0<p\le 1$ and $\XX$ be a Banach space. Then $\HL_p^m(\TT,\XX)=H_p^a(I,\XX)$.
\end{theorem}

\begin{proof}
Set $\sigma=\floor{1/p}-1$. Blasco and Garc\'{\i}a-Cuerva \cite{BGC1987}*{Corollary 3.6} proved that any function $f\in \HL_p^m(\TT,\XX)$ can be decomposed as a sum of atoms, but they considered functions in the unit ball of $L_\infty(I,\XX)$ as atoms. Knowing this, to finish the proof it suffices to show that $L_\infty(I,\XX)\subseteq H_p^a(I,\XX)$ continuously. Taking into account that the $\sigma$th Gram-Schmidt projection $Q_\sigma$ relative to $(p_j)_{j=0}^\infty$ is a bounded operator on $L_\infty(I,\XX)$, the result is a ready consequence of the expansion \eqref{eq:GSExpansion}.
\end{proof}

\subsection{ Banach-valued Hardy spaces of analytic functions.} The conjugate Poisson kernel $(Q_r)_{0\le r <1}$ is defined by $Q_r(t)=\mathbf{Q}(r e^{2\pi i t})$ for all for $t\in\RR$, where
\[
\mathbf{Q}(z)=\Im \enpar{\frac{1+z}{1-z}}.
\]

Let $0<p\le 1$. Given $f\in \Dt'(\TT,\XX)$ we put

\[
M_{p,f}(r) = \enpar{\int_I \norm{P_r*f}_\XX^p + \norm{Q_r*f}_\XX^p}^{1/p}, \quad 0\le r <1
\]
and
\begin{equation}\label{eq:HpTNorm}
\norm{f}_{H_p(\TT,\XX)}=\sup_{0\le r<1} M_{p,f}(r).
\end{equation}
We then define
\[
H_p(\TT,\XX)=\enbrace{f\in\Dt'(\TT,\XX)\colon \norm{f}_{H_p(\TT,\XX)}<\infty}.
\]
Given a function $F\colon\DD\to \XX$ and $0\le r<1$, let $F_r\colon\TT\to \XX$ be the function
\[
F_r(t)= F\enpar{r e^{2\pi i t}}, \quad t\in\RR.
\]

If $\BB$ is a real Banach space, we denote by $\BB^{\,\CC}$ its complexification. For every $f\in \Dt'(\TT,\BB)$ we set
\[
F_f(z)=P_r*f(t)+i Q_r*f(t)\in\BB^{\,\CC}, \quad z=re^{2\pi i t}\in\DD.
\]
Note that $F_f$ is analytic. Besides, there is a geometric constant $C$ such that if $\XX^{\, \CC}$ is the complexification of $\XX$ regarded as a real Banach space, then
\[
\frac{1}{C} M_{p,f}(r) \le \norm{(F_f)_r}_{\XX^{\, \CC}} \le C M_{p,f}(r), \quad 0\le r<1.
\]

If $\YY$ is a complex Banach space and $F\colon\DD\to\YY$ is analytic, then $\norm{F(f)(\cdot)}_\YY^p$ is subharmonic by the corresponding scalar-valued result (see \cite{GCRF1985}*{Chapter I, Corollary 2.1}) and the Hahn-Banach theorem. Hence, the map
\[
r\mapsto\norm{F_r}_{L_p(\TT,\YY)}
\]
is non-decreasing. These observations yield that if we replace the supremum with the lower limit in \eqref{eq:HpTNorm} we obtain an equivalent quasi-norm.

We denote by $H_{p,e}(\TT,\XX)$ and $H_{p,o}(\TT,\XX)$ and $H_{p,0}\enpar{\TT,\XX^{\,\CC}}$ the subspaces of $H_{p}(\TT,\XX)$ consisting of all even distributions, all odd distributions, and all distributions $f$ with $f(\chi_\RR)=0$, respectively. We also put
\[
H_{p,e,0}(\TT,\XX)=H_{p,e}(\TT,\XX) \cap H_{p,0}\enpar{\TT,\XX}.
\]
The symmetry with respect to the origin, here denoted $\St$, is an isometry on $H_{p}(\TT,\XX)$. Hence, the map
\begin{equation}\label{eq:EvenOdd}
\Pt_{oe}=\enpar{\frac{\Id_{\Dt'(\TT,\XX)}+\St}{2}, \frac{\Id_{\Dt'(\TT,\XX)}-\St}{2}}
\end{equation}
gives the isomorphism
\begin{equation}\label{eq:HpTeo}
H_p(\TT,\XX) \simeq H_{p,e}(\TT,\XX) \oplus H_{p,o}(\TT,\XX).
\end{equation}

If $f\in\Ct^{(\infty)}(\TT,\XX)$, then its conjugate function $\widetilde{f}$ given by
\[
\widetilde{f}\colon\RR\to \XX, \quad t \mapsto\lim_{r\to 1^-} Q_r*f(t)
\]
belongs to $\Ct^{(\infty)}(\TT,\XX)$ as well. This fact allows us to define $\Ct(f):=\widetilde{f}$ for $f\in\Dt'(\TT,\XX)$. The mere definition gives that $\Ct$ is an automorfism of $H_{p,0}(\TT,\XX)$ . Besides, it maps $H_{p,e,0}(\TT,\XX)$ onto $H_{p,e}(\TT,\XX)$.

Bourgain \cite{Bourgain1982} stated that $H_1(\TT,\ell_2)$ does not complementably embed in $H_1(\TT)$ and proved that $\Rad(H_1^\delta(I))$ does not complementably embed in $H_1^\delta(I)$. This makes us guess that he was aware of the validity of the following result in the case when $p=1$.

\begin{proposition}\label{prop:HTRad}
Let $0<p\le 1$ and $(\varepsilon_n)_{n=1}^\infty$ be a sequence of Rademacher random variables in a probability space $(\Omega,P)$. Then the map
\[
c_{00}(H_p(\TT)) \to L_1(\Omega,P), \quad (f_n)_{n=1}^\infty \mapsto \sum_{n=1}^\infty f_n \,\varepsilon_n
\]
extends to an isomorphism from $H_p(\TT,\ell_2)$ onto $\Rad(H_p(\TT))$.
\end{proposition}

\begin{proof}
Let $f=(f_n)_{n=1}^\infty\in c_{00}(H_p(\TT))$. For each $r\in[0,1)$ put
\[
H(r)=\int_I \int_\Omega \enpar{\abs{ \sum_{n=1}^\infty \varepsilon_n(\omega) P_r*f_n(t)}^p +
\abs{ \sum_{n=1}^\infty \varepsilon_n(\omega) Q_r*f_n(t)}^p} dP(\omega) dt.
\]
By Khintchine's inequalities, there is a constant $C$ so that
\[
\frac{1}{C} \norm{f}_{H_p(\TT,\ell_2)}^p \le \sup_{0\le r<1} H(r) \le C \norm{f}_{H_p(\TT,\ell_2)}^p.
\]
Now, for $r\in[0,1)$ and $\omega\in\Omega$ define
\[
H_1(\omega,r)=\int_I \enpar{\abs{ \sum_{n=1}^\infty \varepsilon_n(\omega) P_r*f_n(t)}^p +
\abs{ \sum_{n=1}^\infty \varepsilon_n(\omega) Q_r*f_n(t)}^p} dt\]
and
\[
H_2(\omega,r)=\enpar{\int_I \enpar{\abs{ \sum_{n=1}^\infty \varepsilon_n(\omega) P_r*f_n(t)}^2 +
\abs{ \sum_{n=1}^\infty \varepsilon_n(\omega) Q_r*f_n(t)}^2} dt}^{p/2}.
\]
There is another constant $C_0$ such that $H_1\le C_0 H_2$ and $H_2\le C_0 H_1$.
Since, by Fubini's theorem, $H(r) = \int_\Omega H_1(\omega,r)\, dP(\omega)$,
\[
\frac{1}{C C_0} \norm{f}_{H_p(\TT,\ell_2)}^p \le \sup_{0<r<1} \int_\Omega H_2(\omega,r) \, dP(\omega) \le C C_0 \norm{f}_{H_p(\TT,\ell_2)}^p.
\]
$H_2(\omega,\cdot)$ is non-decreasing for each $\omega\in\Omega$, and
\[
\lim_{r\to 1^-} H_2(\omega,r)= \norm{ \sum_{n=1}^\infty \varepsilon_n(\omega) \, f_n}^p_{H_p(\TT)}.
\]
By the monotone convergence theorem,
\[
\sup_{0<r<1} \int_\Omega H_2(\omega,r)\, dP(\omega) = \int_\Omega \norm{ \sum_{n=1}^\infty \varepsilon_n(\omega) \, f_n}^p_{H_p(\TT)} dP(\omega).\qedhere
\]
\end{proof}

The spaces $H_p(\TT,\XX)$ and $H_p^m(\TT,\XX)$ were shown to be the same by Blasco under the assumption that $\XX$ is an UMD space. Recall that a Banach space $\XX$ is said to be a UMD space if there is some $q\in(1,\infty)$ (equivalently, for all $q\in(1,\infty)$) such that the differences of martingales are unconditional sequences in $L_q(I,\XX)$, i.e., there is a constant $\beta$ such that
\[
\norm{ \sum_{n=0}^m \varepsilon_n \enpar{f_n-f_{n-1}}}_{L_q(\mu,\XX)} \le \beta \norm{f_m}_{L_q(\mu,\XX)},
\]
for all martingales $(f_n)_{n=0}^\infty$ over any measure space $(\Omega,\Sigma,\mu)$, all sequences $(\varepsilon_n)_{n=0}^\infty$ of scalars of modulus at most one, and all $m\in\NN_0$, with the convention $f_{-1}=0$.

\begin{theorem}[\cite{Blasco1988}*{Theorem 3.2}]\label{thm:BlascoHardy}
Let $\XX$ be a Banach space and $0<p\le 1$. Then $H_p(\TT,\XX)=H_p^m(\TT,\XX)$, if and only if $\XX$ is a UMD space.
\end{theorem}

Given a complex Banach space $\XX$ and $0<p\le 1$ we define $H_p(\DD,\XX)$ as the quasi-Banach space of all $F\in\Ht(\DD,\XX)$ such that
\[
\norm{F}_{H_p(\DD,\XX)}:=\sup_{0\le r <1} \norm{F_r}_{L_p(\TT,\XX)}<\infty.
\]
These spaces are perhaps the most important and classic of all Hardy spaces. To relate them with other Hardy spaces, we should realize that functions in $H_p(\DD,\XX)$ define distributions in $\Dt'(\TT,\XX)$. According to the authors of \cite{BGC1987} this can be done by adapting the method used in \cite{GCRF1985} for scalar-valued Hardy spaces over $\RR$. Here we present an argument that better suits the particularities of the class of Hardy spaces defined on the torus.

Recall that, roughly speaking, the Banach envelope of a quasi-Banach space $\XX$ is the `closest' Banach space to $\XX$. To give a precise definition, we use a universal property. We say that a Banach space $\widehat{\XX}$ is the Banach envelope of $\XX$ via a map $J\colon \XX\to \widehat{\XX}$ if $J$ is a linear contraction and for any Banach space $\BB$ and any linear contraction $T\colon\XX\to \BB$ there is a linear contraction $S\colon \widehat{\XX}\to \BB$ such that $S\circ J=T$. For more details on Banach envelopes of quasi-Banach spaces, we refer the unfamiliar reader to \cite{AAW2021c}.

\begin{lemma}\label{lem:DualBergman}
Let $0<p\le 1$ and $\XX$ be a complex Banach space. For any $F\in H_p(\DD,\XX)$ there is $\Gamma(F)\in\Dt'(\TT,\XX)$ such that
\begin{equation}\label{eq:DualBergman}
\Gamma(F)(g)=\lim_{s\to 1^-} \int_{I} F_s\, g, \quad g\in \Ct^{(\infty)}(\TT,\CC).
\end{equation}
Moreover, $\Gamma$ is a continuous linear map.
\end{lemma}

\begin{proof}
If $0<p<1$, the Banach envelope of $H_p(\DD,\XX)$ is (via the inclusion map) the vector-valued Bergman space $B_{1,p}(\DD,\XX)$, which consists of all $f\in \Ht(\DD,\XX)$ such that
\[
\int_\DD \norm{f(x+iy)}_\XX \enpar{1-\sqrt{x^2+y^2}}^{1/p-2} \, dx\, dy<\infty
\]
(see \cite{MN2002}). Let $n\in\NN_0$ be such that $n \ge 1/p-1$. Given $F\in B_{1,p}(\DD,\XX)$, let $H\in\Ht(\DD)$ be such that $H^{(n)}=F$ and $H^{(k)}(0)=0$ for $0\le k \le n-1$. Since
\[
\int_{I} F_s g =(2\pi i)^n s^n \int_I H_s g^{(n)},
\]
in order to prove the result for $F$, it suffices to prove it for $H$. It is known (see \cite{Blasco1990}*{Lemma 1.2 and Proposition 1.1}) that $H\in H_\infty(\DD,\XX)$. Since $H_\infty(\DD,\XX)\subseteq H_1(\DD,\XX)$, it suffices to prove the result for $p=1$.

If $F\in H_1(\DD,\XX)$ its Taylor coefficients $\alpha:=(a_n)_{n=0}^\infty$ constitute a sequence in $\ell_\infty(\XX)$. Assume that $F(0)=0$.
Given $\varphi\in\Ct^{(\infty)}(\TT,\CC)$, its Fourier coefficients $(\hat{\varphi}(n))_{n\in\ZZ}$ satisfy
\[
\abs{\hat{\varphi}(n)} \le \frac{1}{4\pi ^2 n^2} \norm{\varphi''}_{L_1(\TT)}, \quad n\not=0.
\]
Therefore,
\[
\lim_{s\to 1^-} \int_{I} F_s \varphi=f(\varphi):=\sum_{n=0}^\infty a_n \hat{\varphi}(n),
\]
and $\abs{f(\varphi)} \le \norm{\alpha}_{\ell_\infty(\XX)} \norm{\varphi''}_{L_1(\TT)} /24$.
\end{proof}

\begin{lemma}\label{lem:ev0}
Let $0<p\le 1$ and $\XX$ be a complex Banach space. Then, the evaluation map
\[
F\mapsto F(0)
\]
is bounded from $H_{p}(\DD,\XX)$ into $\XX$.
\end{lemma}

\begin{proof}
It follows from Lemma~\ref{lem:DualBergman}. Indeed, $F(0)=\Gamma(F)(\chi_\RR)$.
\end{proof}
If we think of vectors in $\XX$ as constant functions, the evaluation maps become projections. Hence, Lemma~\ref{lem:ev0} gives that $H_{p}(\DD,\XX)$ is isomorphic to $\XX\oplus H_{p,0}(\DD,\XX)$, where
\[
H_{p,0}(\DD,\XX)=\enbrace{ F \in H_{p}(\DD,\XX) \colon F(0)=0}.
\]
Let us write down other elementary isomorphisms involving $H_{p}(\DD,\XX)$. The map
\[
F(z)\mapsto F(-z)
\]
defines an automorphism of $H_p(\DD, \XX)$. Thefefore, if $H_{p,e}(\DD,\XX)$ and $H_{p,o}(\DD,\XX)$, denote the subspaces of $H_p(\DD,\XX)$ consisting of all even and all odd functions respectively,
\begin{equation}\label{eq:DiscEvenOdd}
H_p(\DD,\XX) \simeq H_{p,e}(\DD,\XX)\oplus H_{p,o}(\DD,\XX).
\end{equation}
In turn, the map
\[
F(z)\mapsto F(z^2)
\]
defines an isometry from $H_p(\DD, \XX)$ onto $H_{p,o}(\DD,\XX)$. Finally, the map
\[
F(z)\mapsto z F(z)
\]
defines an isometry from $H_{p}(\DD,\XX)$ onto $H_{p,0}(\DD,\XX)$ which restricts to an isometry from $H_{p,e}(\DD,\XX)$ onto $H_{p,o}(\DD,\XX)$. Therefore, using \eqref{eq:DiscEvenOdd},
\begin{equation}\label{eq:HpDSquare}
H_{p}(\DD,\XX) \simeq H_{p,o}(\DD,\XX) \oplus H_{p,o}(\DD,\XX)\simeq H_p(\DD,\XX) \oplus H_p(\DD,\XX).
\end{equation}

We say that a complex Banach space is \emph{regular} if it is isomorphic to the complexification of a real Banach space $\XX$.

\begin{lemma}\label{lem:torodisco}
Let $0<p\le 1$ and $\XX$ be a real Banach space. Then the map
\[
f \mapsto \Lambda(f) \colon \DD \to \XX^{\,\CC}, \quad \Lambda(f)(z)= P_r*f(t) +i Q_r*f(t) , \quad z=re^{2\pi i t}\in\DD,
\]
defines a (real) isomorphism from $H_p(\TT,\XX)$ onto
\[
\HH:=\enbrace{ F \in H_p\enpar{\DD,\XX^{\,\CC}} \colon \Im(F(0))=0}
\]
whose inverse is $\Re(\Gamma)$, where $\Gamma$ is as in Lemma~\ref{lem:DualBergman}.
\end{lemma}

\begin{proof}
Using just the definitions of the spaces yields that $\Lambda$ is an isomorphic embedding from $H_p(\TT,\XX)$ into $\HH$, so we need only see that $G:=\Lambda(\Re(\Gamma(F)))=F$ for all $F\in \HH$. To that end we note that
\[
F(sz)= P_r*\Re(F_s)(t)+i Q_r*\Re(F_s)(t), \quad z=re^{2\pi i t}\in\DD, \, 0\le s<1.
\]
Hence,
\[
G(z)=\lim_{s\to 1^{-}} P_r*\Re(F_s)(t)+i Q_r*\Re(F_s)(t)=\lim_{s\to 1^-} F(sz)=F(z).\qedhere
\]
\end{proof}

Lemma~\ref{lem:torodisco} is optimal from a harmonic analysis point of view as it opens the door to real-analysis criteria to determine whether an analytic function is a member of $H_p(\DD,\XX)$. However, from the perspective of functional analysis it could be claimed that it does not identify a (complex) Banach space to which $H_p\enpar{\DD,\XX^{\,\CC}}$ is isomorphic. Our next result tells us how to obtain complex isomorphisms using Lemma~\ref{lem:torodisco}.

\begin{proposition}\label{prop:toroMD}
Let $0<p\le 1$ and $\YY$ be a regular complex Banach space.
\begin{enumerate}[label=(\roman*), leftmargin=*, widest=ii]
\item\label{it:HLW} The map $S_e$ that assigns to each monomial $z\mapsto x z^n$ the trigonometric monomial $t\mapsto x \cos(2\pi n t)$, $n\in\NN_0$, $x\in\YY$, extends to an isomorphism from $H_p\enpar{\DD,\YY}$ onto $H_{p,e}\enpar{\TT,\YY}$.
\item\label{it:HLWE} The map $S_o$ that assigns to each monomial $z\mapsto x z^n$ the trigonometric monomial $t\mapsto x \sin(2\pi n t)$, $n\in\NN$, $x\in\YY$, extends to an isomorphism from $H_{p,0}\enpar{\DD,\YY}$ onto $H_{p,o}\enpar{\TT,\YY}$.
\end{enumerate}
\end{proposition}

\begin{proof}
Without loss of generality we may assume that $\YY=\XX^{\,\CC}$ for some real Banach space $\XX$. Once we have made clear that evaluating a function at zero defines a bounded operator on $H_p(\DD,\YY)$, and that the conjugation $\Ct$ and the symmetry $\St$ are bounded operators on $H_p(\TT,\XX)$, the proof will be mere linear algebra. Since $\Ct$ maps cosine functions into sine functions, to prove \ref{it:HLWE} it suffices to show that the linear map $\St_e$ is an isomorphism from $H_{p,0}\enpar{\DD,\YY}$ onto $H^m_{p,e,0}\enpar{I,\YY}$. Besides, this isomorphism would also yield \ref{it:HLW}.

Let $\Gamma$ be as in Lemma~\ref{lem:DualBergman}. Then
\[
S_e= \frac{\Gamma +\St\circ \Gamma} {2}.
\]
Given $F\in H_{p,0}\enpar{\DD,\YY}$ we have $\widetilde{\Re(\Gamma(F))}=\Im(\Gamma(F))$, whence, if
\[
T_e=\frac{\Id_{\Dt'(\TT,\XX^{\,\CC})}+\St+ i \Ct +i \St \circ \Ct}{2},
\]
$S_e(F)=T_e( \Re(\Gamma(F)))$ for all $F\in H_{p,0}\enpar{\DD,\YY}$. Since $ \Ct\circ \St=-\St \circ \Ct$, if we identify $\Dt'(\TT,\XX^{\,\CC})$ with $\Dt'(\TT,\XX)\oplus \Dt'(\TT,\XX)$ and $\Pt_{oe}$ is as in \eqref{eq:EvenOdd}, the restriction to $\Dt'(\TT,\XX)$ of $T_e$ is given by
\[
T_e= \enpar{\Id_{\Dt'(\TT,\XX)}, \Ct}\circ \Pt_{oe}.
\]
Therefore, $T_e$ is a real isomorphism from $H_{p,0}\enpar{\TT,\XX}$ onto $H_{p,e,0}\enpar{\TT,\XX^{\,\CC}}$. Applying Lemma~\ref{lem:torodisco} puts an end to the proof.
\end{proof}

\begin{corollary}\label{cor:HTRad}
Let $\YY$ be a regular complex Banach space and let $0<p\le 1$. Then $H_p(\DD,\YY) \simeq H_p(\TT,\YY)$.
\end{corollary}

\begin{proof}
Just combine Proposition~\ref{prop:toroMD}, \eqref{eq:HpDSquare} and \eqref{eq:HpTeo}.
\end{proof}

\subsection{ Banach-valued square dyadic Hardy spaces}
The square-function characterization of scalar-valued dyadic Hardy spaces is due to Burkholder and Gundy \cite{BG1970}. Here we will use a vector-valued extension obtained by Tozoni \cite{Tozoni1995}.

Suppose that $\XX$ is a Banach lattice. Given $f=(f_n)_{n=0}^\infty\in\Mt(I,\XX)$, put
\[
S_m(f)\colon I \to \XX, \quad S_m(f)= \enpar{\sum_{n=0}^m \abs{f_n-f_{n-1}}^2}^{1/2}, \quad m\in\NN_0,
\]
with the convention that $f_{-1}=0$, and define
\begin{equation}\label{eq:squarenorm}
\norm{f}_{H_p^{\delta,s}(I,\XX)}:= \sup_{m\in\NN_0} \norm{S_m(f)}_{L_p(I,\XX)}\in[0,\infty].
\end{equation}
Since $(S_m(f))_{m=0}^\infty$ is a non-decreasing sequence of positive vectors in the quasi-Banach lattice $L_p(I,\XX)$, the sequence $(\norm{S_m(f)}_{L_p(I,\XX)})_{m=0}^\infty$ is a non-decreasing, so we could replace the supremum with the limit in \eqref{eq:squarenorm}.
The \emph{$\XX$-valued square dyadic Hardy space of index $p$} will be the quasi-Banach space
\[
H_p^{\delta,s}(I,\XX)=\enbrace{ f \in \Mt(I,\XX) \colon \norm{f}_{H_p^{\delta,s}(I,\XX)}<\infty}.
\]

\begin{theorem}[see \cite{Tozoni1995}*{Theorem 5.1}]\label{thm:Tozoni}
Let $0<p\le 1$ and $\XX$ be a UMD Banach lattice. Then there is a constant $C$ such that
\[
\frac{1}{C} \norm{f}_{H_p^{\delta,m}(I,\XX)} \le \norm{f}_{H_p^{\delta,s}(I,\XX)} \le C \norm{f}_{H_p^{\delta,m}(I,\XX)}
\]
for all $f\in \Mt(I,\XX)$.
\end{theorem}

\begin{corollary}\label{cor:squareMD}
Let $0<p\le 1$ and $\XX$ be a UMD Banach lattice. Then,
\[
H_p^{\delta,s}(I,\XX)=H_p^{\delta,m}(I,\XX)=\HL_p^{\delta,m}(I,\XX).
\]
\end{corollary}

\begin{proof}
UMD Banach spaces are reflexive and so they have the RNP. Taking this into account, we need only combine Theorem~\ref{thm:Tozoni} with Theorem~\ref{thm:density}.
\end{proof}

\subsection{Banach-valued discrete Hardy spaces.}
Expansions by means of Haar functions allow us to endow $H_p^{\delta,s}(I,\XX)$, $0<p\le 1$, with a lattice structure.

Recall that the $L_2$-normalized Haar system $(h_J)_{J\in\Dt}$ is given by
\[
h_J=\abs{J}^{-1/2} \enpar{ -\chi_{J^-}+\chi_{J^+}},
\]
where $J^-$ and $J^+$ are the left-half and the right-half of $J$, respectively. Given a linear combination $h=\sum_{J\in \Jt} a_J \chi_J$ of indicator functions of dyadic intervals and a dyadic martingale $F$ we put
\[
\enangle{F, h}=\sum_{J\in \Jt} a_J \Ave(F,J).
\]
An arrangement $(J_n)_{n=1}^\infty$ of $\Dt$ is said to be \emph{natural} if $\abs{J_{n+1}}\le \abs{J_{n}}$ for all $n\in\NN$. Given such an arrangement we define $J_0=I$, $h_0=\chi_I$ and $h_n=h_{J_n}$ for all $n\in\NN$. For any dyadic martingale $F=(f_k)_{k=0}^\infty$ we have
\begin{equation*}
f_k=\enangle{F,\chi_I} \chi_I + \sum_{j=0}^m \sum_{J\in \Dt_j} \enangle{F, h_J} h_J
=\sum_{n=0}^{2^k} \enangle{F, h_{J_n}} h_{J_n}.
\quad k\in\NN.
\end{equation*}
Hence, the map
\[
\Ft\colon \Mt(I,\XX) \to \XX^{\NN_0}, \quad F \mapsto \enpar{ \abs{J_n}^{1/p-1/2} \enangle{F,h_n}}_{n=0}^\infty,
\]
is a linear bijection with inverse
\[
\St\colon \XX^{\NN_0} \to \Mt(I,\XX), \, (\gamma_n)_{n=0}^\infty \mapsto \enpar{ \sum_{n=0}^{2^k} \abs{J_n}^{1/2-1/p} \gamma_n h_n}_{k=0}^\infty.
\]
Note that $\abs{h_n} =\abs{J_n}^{-1/2}\chi_{J_n}$ for all $n\in\NN_0$. Therefore, if we put
\[
\norm{\gamma}_{H_p(\Dt,\XX)}=\sup_{k\ge 0}\norm{ \enpar{\sum_{n=0}^{2^k} \abs{\gamma_n}^2 \abs{J_n}^{-2/p} \chi_{J_n}}^{1/2}}_{L_p(I,\XX)}, \quad \gamma=(\gamma_n)_{n=0}^\infty,
\]
then $\norm{\St(\gamma)}_{H_p^{\delta,s}(I,\XX)}=\norm{\gamma}_{H_p(\Dt,\XX)}$ for all $\gamma=\in\XX^{\NN_0}$.
So,
\begin{equation}\label{eq:squarediscrete}
H_p^{\delta,s}(I,\XX) \simeq H_p(\Dt,\XX):=\enbrace{ \gamma\in \XX^{\NN_0} \colon \norm{\gamma}_{H_p(\Dt,\XX)}<\infty}
\end{equation}
isometrically. If $\XX$ is a K\"othe space then given $\gamma=(\gamma_n)_{n=0}^\infty\in\XX^{\NN_0}$ we can define
\[
T(\gamma)
:=\enpar{ \sum_{n=0}^\infty \abs{\gamma_n}^2 \abs{J_n}^{-2/p} \chi_{J_n}}^{1/2}
\]
as a function that assigns to each point of $I$ a measurable function that possibly takes the value $\infty$. Since $L_p(I,\XX)$ inherits from $\XX$ the property of being a K\"othe space,
\[
\norm{\gamma}_{H_p(\Dt,\XX)}= \norm{T(\gamma)}_{L_p(I,\XX)}, \quad \gamma\in\XX^{\NN_0}.
\]

These discrete Hardy spaces $H_p(\Dt,\XX)$ are closely related to Triebel-Lizorkin spaces. Given $p\in(0,\infty)$ and $q\in(0,\infty]$, we define the space $\TL_{p,q}$ as the quasi-Banach lattice associated with the function quasi-norm over $\NN_0$ defined as
\[
\tau_{p,q}(\gamma) = \norm{\enpar{ \sum_{n=0}^\infty \gamma_n^q \abs{J_n}^{-q/p} \chi_{J_n} }^{1/q}}_{L_p(I)}, \quad
\gamma=(\gamma_n)_{n=0}^\infty,
\]
with the standard modifcation if $q=\infty$. Note that neither $H_p(\Dt,\XX)$ nor the discrete Triebel-Lizorkin spaces $\TL_{p,q}$ depend on the particular natural arrangement of $\Dt$ chosen. The lattice $\TL_{p,q}$ is $p$-convex and $q$-concave. Therefore, if $0<p\le 1\le q\le 2 \le r\le\infty$ and $\XX$ is lattice $q$-convex and lattice $r$-concave,
\[
\frac{1}{ M_{(r)}(\XX) } \tau_{p,r} \enpar{\norm{\gamma}_\XX}\le \norm{\gamma}_{H_p(\Dt,\XX)} \le M^{(q)}(\XX) \tau_{p,q} \enpar{\norm{\gamma}_\XX},
\quad \gamma\in \XX^{\NN_0},
\]
so that
\[\TL_{p,q}(\XX) \subseteq H_p(\Dt,\XX)\subseteq \TL_{p,r}(\XX)\] continuously. In the particular case that $\XX$ is a Hilbert space, then
\begin{equation}\label{eq:TLIso}
H_p(\Dt,\XX)=\TL_{p,2}(\XX), \quad 0<p\le 1,
\end{equation}
isometrically.

Suppose that $\XX$ is a UMD space. In light of Theorems~\ref{thm:BGC}, \ref{thm:Anso} and \ref{thm:BGC+Mol}, Corollaries~\ref{cor:HTRad} and \ref{cor:squareMD}, and \eqref{eq:squarediscrete}, in order to prove that all $\XX$-valued Hardy spaces that we introduced are isomorphic it remains to see that $H_p^{\delta,a}(I,\XX)$ and $H_p^{a}(I,\XX)$ are isomorphic. Maurey \cite{Maurey1980} and Carleson \cite{Carleson1980} obtained this result independently for $p=1$ in the scalar valued case. Later on, Wojtaszczyk \cite{Woj1984} extended the isomorphism to the whole range $0<p\le 1$. Taking into account the equivalences between different classes of Hardy spaces, this result can be written in the following form
\begin{equation}\label{eq:WojIso}
H_p(\TT) \simeq\TL_{p,2}, \quad 0<p\le 1.
\end{equation}
We use this isomorphism and \eqref{eq:TLIso} to come full circle when the target space is a Hilbert space.

\begin{theorem}\label{thm:HpDTL}
Given $0<p\le 1$, $H_p(\DD,\ell_2)\simeq\TL_{p,2}(\ell_2)$.
\end{theorem}

\begin{proof}
By \eqref{eq:WojIso}, $\Rad(H_p(\TT)) \simeq \Rad(\TL_{p,2})$. Combining this isomorphism with Proposition~\ref{prop:RadL2}, Proposition~\ref{prop:HTRad} and Corollary~\ref{cor:HTRad} yields the desired result.
\end{proof}

\section{Unconditional basic sequences in quasi-Banach lattices}\label{sect:UBQB}\noindent
Let us start this section by introducing some more terminology. The closed linear span of a subset $V$ of a quasi-Banach space $\XX$ will be denoted by $[V]$. A countable family $\XB=(\xx_n)_{n \in \Nt}$ in $\XX$ is an \emph{unconditional basic sequence} if for every $f\in[\xx_n \colon n \in \Nt]$ there is a unique family $(a_n)_{n \in \Nt}$ in $\FF$ such that the series $\sum_{n \in \Nt} a_n \, \xx_n$ converges unconditionally to $f$. If we additionally have $[\xx_n \colon n \in \Nt]=\XX$, then $\XB$ is an \emph{unconditional basis} of $\XX$.

If $\XB$ is an unconditional basis of $\XX$, the map
\[
\Fou\colon\XX\to \FF^{\Nt},\quad f=\sum_{n \in \Nt} a_n\, \xx_n \mapsto (\xx_n^*(f))_{n \in \Nt} = (a_n)_{n \in \Nt}
\]
will be called the \emph{coefficient transform} with respect to $\XB$, and the functionals $(\xx_n^*)_{n \in \Nt}$ the \emph{coordinate functionals} of $\XB$.

In keeping with current usage we will write $c_{00}(\Nt)$ for the set of all familes $(a_n)_{n\in \Nt}$ in $\FF^{\Nt}$ such that $\abs{\{n\in \Nt \colon a_n\not=0\}}<\infty$. If $\XB=(\xx_n)_{n \in \Nt}$ is an unconditional basic sequence in $\XX$, there is a constant $K\ge 1$ such that
\begin{equation}\label{eq:ubslattice}
\norm{ \sum_{n \in \Nt} b_n \, \xx_n} \le K \norm{ \sum_{n \in \Nt} a_n \, \xx_n}
\end{equation}
for all $(a_n)_{n\in \Nt}\in c_{00}(\Nt)$ and all $(b_n)_{n\in \Nt}$ in $\FF^{\Nt}$ with $\abs{b_n} \le\abs{a_n}$ for all $n\in\Nt$ (see \cite{AABW2021}*{Theorem 1.10}). If inequality \eqref{eq:ubslattice} is satisfied for a given constant $K$ we say that $\XB$ is $K$-unconditional.

We infer from \eqref{eq:ubslattice} that an unconditional basis $\XB$ of $\XX$ induces a lattice structure on $\XX$. Indeed, if for each $f=(a_n)_{n\in\Nt}$ in $\FF$ we define
\begin{equation*}
\norm{f}=\sup\enbrace{ \norm{ \sum_{n \in \Nt} b_n \, \xx_n}\colon (b_n)_{n\in \Nt}\in c_{00}(\Nt), \, \abs{b_n} \le\abs{a_n} \, \forall \, n\in\Nt},
\end{equation*}
then $\norm{\cdot}$ is a quasi-norm on $\Fou(\XX)$, and $\Fou(\XX)$ is a quasi-Banach lattice called the \emph{quasi-Banach lattice associated with $\XB$.} when equipped with this quasi-norm and the standard ordering The map $\Fou$ is an isomorphism from $\XX$ onto $\Fou(\XX)$. In fact, $\norm{\Fou} \le K$ and $\norm{\Fou^{-1}} \le 1$. Conversely, given a quasi-Banach lattice $\LL$ over $\Nt$, the canonical unit vector system $\Et_{\Nt}:=(\ee_n)_{n\in \Nt}$ of $\FF^{\Nt}$, defined for $n\in \Nt$ by $\ee_n=(\delta_{n,m})_{m\in \Nt}$, where $\delta_{n,m}=1$ if $n=m$ and $\delta_{n,m}=0$ otherwise, is a $1$-unconditional basic sequence of $\LL$. If $\Et_{\Nt}$ is \emph{normalized} in $\LL$, i.e., $\norm{\ee_n}=1$ for all $n\in\Nt$, we say that $\LL$ is a sequence space. If $c_{00}(\Nt)$ is dense in $\LL$, so that $\Et_\Nt$ is an unconditional basis of $\LL$, we say that $\LL$ is \emph{minimal}. The coordinate functionals of the unit vector system of a minimal lattice $\LL$ over $\Nt$ are $(\ee_n^*|_\LL)_{n\in\Nt}$, where for each $n\in\Nt$ $\ee_n^*\colon \FF^{\Nt}\to \FF$ is the canonical coordinate projection defined by
\[
(a_k)_{k\in\Nt} \mapsto a_n.
\]

\subsection{Building lattices from families of quasi-Banach lattices}
Given a countable family $\YB=(\yy_m)_{m\in\Mt}$ in a quasi-Banach lattice $\LL$ we define
\[
\rho_{\YB} \colon [0,\infty)^\Mt \to[0,\infty] , \quad (a_m)_{m\in \Mt} \mapsto
\sup_{\substack{A\subseteq \Mt\\ \abs{A}<\infty}} \norm{ \enpar{\sum_{m\in A} \abs{a_m}^2 \abs{\yy_m}^2}^{1/2}}.
\]
The map $\rho_{\YB}$ is a function quasi-norm over $\Mt$ with the Fatou property. Hence,
\[
\LL_\YB=\{ f\in \FF^\NN \colon \rho_{\YB}(f)<\infty\}
\]
is a quasi-Banach lattice. One could say that passing from $\YB$ to $\LL_\YB$ is an \emph{latticefication} process (see \cite{BCLT1985}). We define it here because we cannot guarantee that unconditionality is preserved when carrying out some natural re-shapings of families in quasi-Banach lattices.

\subsection{Complemented families in quasi-Banach lattices}
An unconditional basic sequence $\YB=(\yy_m)_{m\in \Mt}$ in a quasi-Banach space $\XX$ is said to be \emph{complemented} if its closed linear span $\YY= [\YB]$ is a complemented subspace of $\XX$, i.e., there is a bounded linear map $P\colon\XX\to\YY$ with $P|_\YY=\Id_\YY$. An unconditional basic sequence $\YB$ is complemented in $\XX$ if and only if there exists a sequence $\YB^*=(\yy_m^*)_{m\in \Mt}$ in $\XX^*$ such that $\yy_m^*(\yy_n)=\delta_{m,n}$ for every $(m,n)\in \Mt^2$ and a the map $P\colon\XX\to \XX$ given by
\begin{equation}\label{eq:projCUBS}
P(f)=\sum_{m\in \Mt} \yy_m^*(f) \, \yy_m, \quad f\in\XX,
\end{equation}
is bounded and linear. The functionals $\YB^*$ are called the \emph{projecting functionals} for $\YB$. We have
\[
K[\YB,\YB^*]=\sup_{ \substack{ \abs{a_m} \le 1 \\ \norm{f} \le 1}}\norm{ \sum_{m\in \Mt} a_m \, \yy_m^*(f) \, \yy_m}<\infty.
\]
This observation immediately yields the following result.

\begin{lemma}\label{lem:bilinear}
Let $\YB=(\yy_m)_{m\in \Mt}$ be a complemented basic sequence in a quasi-Banach space $\XX$ with projecting functionals $\YB^*=(\yy_m^*)_{m\in \Mt}$. Then the map
\[
(f,f^*) \mapsto \enpar{ f^*(\yy_m) \, f^*(\yy_m)}_{m\in \Mt}, \quad f\in \XX, \, f^*\in\XX^*
\]
defines a bilinear map into $\ell_1(\Mt)$ whose norm is bounded by $K[\YB,\YB^*]$.
\end{lemma}

Suppose that $\XB=(\xx_n)_{n \in \Nt}$ and $\YB=(\yy_n)_{n \in \Nt}$ are (countable) families of vectors in quasi-Banach spaces $\XX$ and $\YY$, respectively. We say that $\XB$ $C$-\emph{dominates} $\YB$ if there is a linear map $T$ from $[\XB]$ into $\YY$ with $T(\xx_n)=\yy_n$ for all $n \in \Nt$ such that $\norm{T}\le C$. If $T$ is an isomorphic embedding and $\max\{ \norm{T}, \norm{T^{-1}}\}\le C$, we say that $\XB$ and $\YB$ are $C$-\emph{equivalent}. In all cases, if the constant $C$ is irrelevant, we drop it from the notation.

If $\XB=(\xx_n)_{n\in\Nt}$ is an unconditional basic sequence in a quasi-Banach space $\XX$, then
\[
\norm{\sum_{n\in\Nt} a_n \,\xx_n} \approx \Ave_{\varepsilon_n=\pm 1} \norm{\sum_{n\in\Nt} \varepsilon_n\, a_n \,\xx_n}, \quad (a_n)_{n\in\Nt} \in c_{00}(\Nt).
\]
Hence, by Lemma~\ref{lem:MaureyQB}, if $\XX$ is an $L$-concave quasi-Banach lattice,
\begin{equation}\label{eq:averagingUnc}
\norm{\sum_{n\in\Nt} a_n \,\xx_n} \approx \norm{\enpar{ \sum_{n\in\Nt} \abs{a_n}^2 \abs{\xx_n}^2}^{1/2} }, \quad (a_n)_{n\in\Nt} \in c_{00}(\Nt).
\end{equation}

In Banach lattices \eqref{eq:averagingUnc} still holds if we drop the condition that $\XX$ has nontrivial concavity and assume instead that $\XB$ is complemented (see \cite{LinTza1979}*{Proposition~1.d.6 and subsequent Remark~1}). We would like to rely on the validity of an equivalence like \eqref{eq:averagingUnc} in our more general setting but its proof passes through duality so it has to be redone. Theorem~\ref{thm:KaltonVectorizes} will come to our aid to show \eqref{eq:averagingUnc} from scratch in spaces without local convexity. Besides, for our purposes it will be crucial to replace $(\xx_n)_{n\in\Nt}$ with a suitable family $(\yy_n)_{n\in\Nt}$ on the right-hand side of \eqref{eq:averagingUnc}.

\begin{theorem}\label{thm:LTImproved}
Let $\LL$ be an $L$-convex quasi-Banach lattice with $L$-convexity constant $\varepsilon$. Let $\XB=(\xx_n)_{n\in\Nt}$ be a complemented unconditional basic sequence in $\LL$ with projecting functionals $\XB^*=(\xx_n^*)_{n\in\Nt}$. Let $\YB=(\yy_n)_{n\in\Nt}$ be a sequence in $\LL$ with
\[
\abs{\yy_n} \le D \abs{\xx_n}, \quad n\in\Nt,
\]
for some constant $D$, and
\[
\delta:=\inf_{n\in\Nt} \abs{\xx_n^*(\yy_n)}>0.
\]
Then $\XB$ is $C$-equivalent to the unit vector system of $\LL_\YB$, where $C$ depends only on $\varepsilon$, $\delta$, $D$ and $K:=K[\XB,\XB^*]$.
\end{theorem}

\begin{proof}
If $\XX$ is the quasi-Banach lattice over $\Nt$ associated with $\XB$, the linear operator $P\colon \LL \to \XX$ given by
\[
f\mapsto (\xx_n^*(f))_{n\in\Nt}
\]
satisfies $\norm{P}\le K$. In turn, the linear operator $S\colon \XX \to \LL$ defined by $\ee_n\mapsto \xx_n$ for all $n\in\Nt$ is a contraction. By Lemma~\ref{thm:LCNatural}, $\XX$ is $L$-convex with an $L$-convexity constant that depends only on $K$ and $\varepsilon$. By Theorem~\ref{thm:KaltonVectorizes}, there are constants $C_1$ and $C_2$ depending only on $K$ and $\varepsilon$ such that
\begin{equation}\label{eq:VectorA}
\norm{\enpar{\sum_{j\in J} \enpar{\abs{\xx_n^*(f_j)}^2}^{1/2}}_{n\in\Nt}}\le C_1 \norm{\enpar{\sum_{j\in J} \abs{f_j}^2}^{1/2}}
\end{equation}
for any finite family $(f_j)_{j\in J}$ in $\LL$, and
\begin{equation}\label{eq:VectorB}
\norm{\enpar{\sum_{j\in J} \abs{\sum_{n\in\Nt} a_{j,n}\, \xx_n}^2}^{1/2}} \le C_2 \norm{\enpar{\sum_{j\in J} \enpar{\abs{a_{j,n}}^2}^{1/2}}_{n\in\Nt}}
\end{equation}
for any finite family $((a_{j,n})_{n\in\Nt})_{j\in J}$ in $c_{00}(\Nt)$. Given $(a_n)_{n\in\Nt} \in c_{00}(\Nt)$, set $J=\{n\in\Nt \colon a_n\not=0\}$. Applying \eqref{eq:VectorA} to the families $(a_n\, \xx_n)_{n\in J}$ and $(a_n\, \yy_n)_{n\in J}$ gives that the unit vector system of $\LL_\XB$ $C_1$-dominates the unit vector system of $\XX$, and the unit vector system of $\LL_\YB$ $(C_1/\delta)$-dominates the unit vector system of $\XX$. In turn, Applying \eqref{eq:VectorB} to the family $(a_n\, \ee_n)_{n\in J}$ gives that the unit vector system of $\XX$ $C_2$-dominates the unit vector system of $\LL_\XB$.

It is clear that the unit vector system of $\LL_\XB$ $D$-dominates the unit vector system of $\LL_\YB$. By construction, the unit vector system of $\XX$ dominates $\XB$, and $\XB$ $K$-dominates the unit vector system of $\XX$. Summing up, we obtain that the unit vector system of $\LL_\YB$ $(C_1/\delta)$-dominates $\XB$, and $\XB$ $(DKC_2)$-dominates the unit vector system of $\LL_\YB$.
\end{proof}

\begin{corollary}\label{cor:LTImproved}
Let $\LL$ be an $L$-convex quasi-Banach lattice with $L$-convexity constant $\varepsilon$. Let $\XB=(\xx_n)_{n\in\Nt}$ be a complemented unconditional basic sequence in a $\LL$ with projecting functionals $\XB^*=(\xx_n^*)_{n\in\Nt}$. Then there is a constant $C$ depending only on $\varepsilon$ and $K:=K[\XB,\XB^*]$ such that
\[
\frac{1}{C}\norm{\sum_{n\in\Nt} a_n \,\xx_n} \le \norm{\enpar{ \sum_{n\in\Nt} \abs{a_n}^2 \abs{\xx_n}^2}^{1/2} } \le C\norm{\sum_{n\in\Nt} a_n \,\xx_n}
\]
for all $ (a_n)_{n\in\Nt} \in c_{00}(\Nt)$.
\end{corollary}

\begin{proof}
Just apply Theorem~\ref{thm:LTImproved} with $\YB=\XB$.
\end{proof}

\section{A method for obtaining uniqueness of unconditional basis}\label{sect:AAMethod}\noindent
Given $0<a<b<\infty$, we say that a subset $A$ of a quasi-Banach space $\XX$ is \emph{$(a,b)$-semi-normalized}, if $a\le \norm{f}\le b$ for all $f\in\XX$. We say that a family $\XB=(\xx_n)_{n\in\Nt}$ in a quasi-Banach space $\XX$ is \emph{permutatively $C$-equivalent} to a family $\YB=(\yy_m)_{n\in \Mt}$ in a quasi-Banach space $\YY$, $1\le C<\infty$, and we write $\XB\sim\YB$, if there is a bijection $\pi\colon \Nt\to \Mt$ such that $\XB$ and $(\yy_{\pi(n)})_{n \in \Nt}$ are $C$-equivalent. As before, if the constants involved are irrelevant, we drop them from the notation. In this terminology, $\XX$ has a \emph{unique unconditional basis} if it has a normalized unconditional basis $\XB$ and any other normalized (or, equivalently, semi-normalized) unconditional basis is permutatively equivalent to $\XB$.

The handling of direct sums of bases, in particular powers of bases, has proven to be a useful tool for the study of the unconditional basis structure of Banach and quasi-Banach spaces. Given a quasi-Banach lattice $\LL$ on $\Jt$, and a family $(\XX_j, \Vert\cdot\Vert_{\XX_{j}})_{j\in\Jt}$ of (possibly repeated) quasi-Banach spaces with moduli of concavity uniformly bounded, the space
\[
\XX:=\enpar{\oplus_{j\in\Jt} \XX_j}_\LL=\enbrace{f=(f_j)_{j\in\Jt}\in\prod_{j\in\Jt} \XX_j\colon (\norm{ f_j}_{\XX_j})_{j\in\Jt} \in \LL},
\]
is a quasi-Banach space equipped with the obvious quasi-norm. If each space $\XX_j$ is a quasi-Banach lattice then $\XX$ is a quasi-Banach lattice with the obvious ordering. Moroever, if $\LL$ is lattice $p$-convex (resp., $p$-concave) and there is a constant $C$ such that $M_{(p)}(\XX_j)\le C$ (resp., $M^{(p)}(\XX_j)\le C$) for all $j\in\Jt$, then $\XX$ is lattice $p$-convex (resp., $p$-concave). If each space has a $K$-unconditional basis $\XB_j=(\xx_{j,n})_{n\in \Nt_j}$, $1\le K<\infty$, then the family $\oplus_{j\in\Jt} \XB_j$ constructed on the set
\begin{equation*}
\Nt= \bigcup_{j\in \Jt} \{j\} \times \Nt_j
\end{equation*}
by means of $\xx_{j,n}=(\xx_{j,n,i})_{i\in \Jt}$, where
\[
\xx_{j,n,i}=\begin{cases} \xx_{i,n}& \text{ if }i=j, \\ 0 & \text{ otherwise,}
\end{cases}
\]
is an unconditional basis of $\XX$.

If $F$ is a finite set and, for each $i\in F$, $\XX_i$ is a quasi-Banach space, the specific lattice structure $\LL$ over $F$ we use to built the quasi-Banach space $(\oplus_{i\in F} \XX_i)_\LL$ is, in the isomorphic theory, irrelevant. Here, we put
\[
\norm{f}=\max_{i\in F} \norm{f_i}\quad f=(f_i)_{i\in F}\in \bigoplus_{i\in F} \XX_i.
\]
If, for some $s\in\NN$,
\[
F=\NN[s]:=\enbrace{i\in \ZZ \colon 1\le i \le s}
\]
and $\XX_i=\XX$ for all $i\in F$, the resulting quasi-Banach space is denoted $\XX^s$ and called the \emph{$s$-fold product} of $\XX$. Similarly, if $\XB$ is a family in $\XX$, we denote by $\XB^s$ the corresponding family in $\XX^s$, and we call it the $s$-fold product of $\XB$. We will refer to the $2$-fold product of a basis or a space as its \emph{square}.

Another important notion is that of subbasis. A subfamily of a family $\XB=(\xx_n)_{n\in\Nt}$ in $\XX$ is a family $\YB=(\xx_n)_{n\in\Mt}$ for some subset $\Mt\subseteq\Nt$. If $\XB$ is a $K$-unconditional basis, $1\le K<\infty$ so is $\YB$, which we then call a \emph{subbasis} of $\XB$. Moreover, if $\XB$ is complemented with projecting functionals $\XB^*=(\xx_n^*)_{n\in\Nt}$, then $\YB$ is complemented with projecting functionals $\YB^*=(\xx_n^*)_{n\in\Mt}$, so that $K[\YB,\YB^*]\le K[\XB,\XB^*]$.

\begin{definition}
Let $\XX$ be a quasi-Banach space with a normalized unconditional basis $\XB$. We say that $\XB$ is \emph{power universal} if for every subbasis $\YB$ of $\XB$ and every $K$-complemented normalized unconditional basic sequence $\ZB$ in $[\YB]$, $1\le K<\infty$, there are $s=s(K,\XB)\in\NN$ and $C=C(K,\XB)\in[1,\infty)$ such that $\ZB$ is permutatively $C$-equivalent to a subbasis of $\YB^s$.
\end{definition}

Note that in the definition of power universality we can weaken the assumption that $\ZB$ is normalized and just assume that it is $(a,b)$-semi-normalized; in doing so, $C$ also depends on $a$ and $b$. Note also that subbases inherit the property of being power universal.

Being $\XB$ power universal does not seem to guarantee that $\XX$ has a unique unconditional basis. Notwithstanding, we take advantage of the recent advances in the theory to provide a mild condition on $\XB$ that closes the gap.

\begin{lemma}\label{lem:PUvsUTAP}
Let $\XX$ be a quasi-Banach space with a normalized unconditional basis $\XB$. Suppose that $\XB$ is power universal and permutatively equivalent to its square. Then $\XX$ has a unique unconditional basis.
\end{lemma}

\begin{proof}
Let $\YB$ be a normalized basis of $\XX$. Then there is $s\in\NN$ such that $\YB$ is permutatively equivalent to a subbasis of $\XB^s$, which, in turn, is permutatively equivalent to $\XB$. This allows us to swap the roles of $\XB$ and $\YB$ and thus claim that $\XB$ is permutatively equivalent to a subbasis of $\YB^t$ for some $t\in\NN$. The power $\XB^t$ of $\XB$ is permutatively equivalent to a subbasis of $\YB^t$. Hence, by the power-cancelling principle for uncondional bases (see \cite{AlbiacAnsorena2022b}*{Theorem 2.3}) $\XB$ is permutatively equivalent to a subbasis of $\YB$. By the Schr\"oder-Bernstein principle for unconditional bases (see \cite{Wojtowicz1988}*{Corollary 1}), $\XB$ and $\YB$ are permutatively equivalent.
\end{proof}

Lemma~\ref{lem:PUvsUTAP} brings our attention to the convenience of obtaining conditions that ensure that a given direct sum of bases is permutatively equivalent to it square. In this regard, we refer the reader to \cite{AlbiacAnsorena2022c}*{Lemma~3.4}. Besides, as it applies to Hardy spaces, we bring up a generalization of a result from \cite{AlbiacAnsorena2023} for studying some aspects related to the unconditional structure of Tsirelson's space. Given a one-to-one map $\pi\colon\Nt\to\Jt$ we define the \emph{shift} $S_\pi\colon \FF^\Nt \to \FF^\Jt$ by
\[
(a_n)_{n\in\Nt}\mapsto (b_j)_{j\in \Jt}, \quad b_j=\begin{cases} a_n & \mbox{ if }j=\pi(n),\\ 0 & \mbox{ otherwise.} \end{cases}
\]
We say that a sequence space over $\NN$ is \emph{square-stable} if there is a one-to-one map $\pi\colon \{1,2\}\times \NN\to\NN$ such that $\pi(i,\cdot)$ is increasing, $i=1$, $2$, and $S_\pi$ restricts to an isomorphic embedding from $\LL^2$ into $\LL$.

A sequence space $\LL$ over $\NN$ is said to be \emph{subsymmetric} if the map $S_\pi$ defines an isomorphic embedding for every increasing map $\pi\colon\NN\to\NN$. While any subsymmetric sequence space is square-stable, Tsirelson's space $\Ts$ witnesses that the converse does not hold.

\begin{proposition}\label{prop:DSSquare}
Let $\LL$ be a square-stable sequence space. For each $j\in\NN$, let $\XX_j$ be a quasi-Banach space with an $K$-unconditional basis $\XB_j$, $1\le K<\infty$. Assume that there is a constant $C$ such that $\XB_j$ is permutatively $C$-equivalent to a subbasis of $\XB_k$ for all $(j,k)\in\NN^2$ with $j\le k$.
Then $(\bigoplus_{j=1}^\infty \XB_j)_\LL$ is permutativety equivalent to its square.
\end{proposition}

\begin{proof}
The argument used to prove \cite{AlbiacAnsorena2023}*{Proposition 4.1}, where it is assumed that $\pi$ is a bijection, works in this slightly more general setting.
\end{proof}

\section{Hilbert-valued lattices with a unique unconditional basis}\label{sect:HVLUTAP}\noindent
Theorem~\ref{thm:HpDTL} yields an atomic lattice structure for $H_p(\DD,\ell_2)$, $0<p\le 1$. Before proving that this structure is unique when $p<1$, we devote a few lines to show that the result does not hold when $p=1$. Indeed, for any $0<p<\infty$ and $0<q\le \infty$, the unit vector system $(\ee_j)_{j=0}^\infty$ of $\FF^{\NN_0}$ is a democratic basis of $\Ts_{p,q}$ with fundamental function equivalent to $\Phi_p:=(m^{1/p})_{m=1}^\infty$, that is,
\begin{equation}\label{eq:demlp}
\norm{\sum_{j\in A} \ee_j}_{\Ts_{p,q}} \approx \abs{A}^{1/p}, \quad A\subseteq \NN_0, \, \abs{A}<\infty,
\end{equation}
(see \cite{AABW2021}*{Proposition 11.13}). Bearing in mind that $\TL_{1,2}$ is lattice isomorphic to its square, we infer that if a subbasis of the canonical basis of $\Ts_{1,2}(\ell_2)$ is democratic with fundamental function equivalent to $\Phi_1$, then it is equivalent to a subbasis of the canonical basis of $\Ts_{1,2}$. It is known \cite{DHK2006} that $\Ts_{1,2}$ has an unconditional basis $\XB$ with fundamental function equivalent to $\Phi_1$ which is nonequivalent to any subbasis of the canonical basis of $\Ts_{1,2}$. Therefore, the direct sum of $\XB$ and the the canonical basis of $\Ts_{1,2}(\ell_2)$, which is an unconditional basis of an space isomorphic to $\Ts_{1,2}(\ell_2)$, is not equivalent to the canonical basis of $\Ts_{1,2}(\ell_2)$.

To prove Theorem~\ref{thm:main}, we will take advantage of the fact that $H_p(\DD)$ has a strongly absolute lattice structure. The term strongly absolute basis was coined in \cite{KLW1990}. Here we transfer this idea to atomic lattices. Loosely speaking, a sequence space $\LL$ over a countable set $\Nt$ is strongly absolute if $\LL\subseteq \ell_1(\Nt)$ but it is far from $\ell_1(\Nt)$, in the sense that if a set $V\subseteq \LL$ satisfies
\[
\norm{f}\le C \norm{f}_1, \quad f\in V,
\]
for some constant $C$, then there is a further constant $D$ such that
\[
\norm{f}_1\le D \norm{f}_\infty, \quad f\in V.
\]
We can formalize this feature of some non-locally convex sequence spaces by saying that $\LL$ is \emph{strongly absolute} if for every $\varepsilon>0$ there is a constant $C\in(0,\infty)$ such that
\begin{equation*}
\norm{f}_1 \le\max \enbrace{ C \norm{f}_\infty , \varepsilon \norm{f}}, \quad f\in\LL.
\end{equation*}

We refer the reader to \cite{AlbiacAnsorena2022c} for a systematic study of strongly absolute bases and spaces.

For further reference, we record an elementary duality result for vector-valued quasi-Banach lattices.

\begin{lemma}\label{lem:VVLDualEnvelope}
Let $\LL$ be a sequence space over a countable set $\Nt$. Assume that $\LL\subseteq\ell_1(\Nt)$, and let $C$ be the norm of the inclusion map. Let $(\XX_n)_{n\in\Nt}$ be a family of quasi-Banach spaces with modulus of concavity uniformly bounded. Then:
\begin{enumerate}[label=(\roman*), leftmargin=*, widest=ii]
\item The Banach envelope of $\XX$ is $C$-isomorphic to $\enpar{\oplus_{n\in\Nt} \widehat{\XX_n}}_{\ell_1}$ via the map
\[
(f_n)_{n\in\Nt} \mapsto (J_{\XX_n} (f_n))_{n\in\Nt}.
\]
\item\label{it:VVLDual} The dual space of $\XX:=(\oplus_{n\in\Nt} \XX_n)_\LL$ is $C$-isomorphic to $\YY:=(\oplus_{n\in\Nt} \XX_n^*)_{\ell_\infty}$ via the dual pairing
\[
\enangle{f^*,g} =\sum_{n\in\Nt} f_n^*(f_n), \quad f^*= (f_n^*)_{n\in\Nt}\in\YY, \, f= (f_n)_{n\in\Nt}\in\XX.
\]
\end{enumerate}
\end{lemma}

\begin{theorem}\label{thm:Ll2PU}
Let $\LL$ be an strongly absolute $L$-convex sequence space over a countable set $\Jt$. Then the unit vector system of $\LL(\ell_2)$ is power universal.
\end{theorem}

\begin{proof}
Let us use the convention that $\ell_2^\infty=\ell_2$ and $\ell_2^0=\{0\}$. We must prove that for every map $\varphi\colon\NN\to \NN\cup\{0,\infty\}$ and every complemented normalized unconditional basic sequence $(\yy_m)_{m\in\Mt}$ with projecting functionals $(\yy_m^*)_{m\in\Mt}$ in
\[
\XX:=\enpar{ \bigoplus_{j\in\Jt} \ell_2^{\varphi(j)}}_\LL
\]
there are $C\in[1,\infty)$, $s\in\NN$ and $\psi\colon\NN\to \NN\cup\{0,\infty\}$ with $\psi\le s \varphi$, in such a way that $\XB$ is $C$-equivalent to the unit vector system of $\enpar{ \bigoplus_{j\in\Jt} \ell_2^{\psi(j)}}_\LL$, where $s$ and $C$ only depend on $\LL$ and $K:=K[\XB,\XB^*]$.

For each $j\in\Nt$ let
\[
L_j\colon\XX_j:=\ell_2^{\varphi(j)} \to \XX, \quad L_j'\colon \XX_j^*\to \XX':=\enpar{ \bigoplus_{j\in\Jt} \XX_j^*}_{\ell_\infty(\Nt)}
\]
be the canonical lifting, and let $P_j\colon\XX\to \XX_j$, $P_j'\colon\XX'\to \XX_j$ be the canonical projections. Let $S\colon \XX^*\to \XX'$ be the isomorphism provided by Lemma~\ref{lem:VVLDualEnvelope}~\ref{it:VVLDual}. For $m\in\Mt$ put
\[
f_m= (f_{j,m})_{j\in\Jt}=\enpar{P_j'( S(\yy_m^*)) (P_j(\yy_m))}_{j\in\Jt}.
\]
Since $\norm{\yy_m^*}\le K$ for all $m\in\Mt$,
\begin{equation}\label{eq:LargeCoef}
\abs{f_{j,m}} \le K \norm{S} \norm{P_j(\yy_m)}, \quad j\in\Jt, \, m\in\Mt.
\end{equation}
Consequently,
\[\norm{f_m}_\LL\le K \norm{S} \norm{\yy_m}= K\norm{S},\quad m\in\Nt.
\] In turn, since
\[
1=\yy_m^*(\yy_m)=\sum_{j\in \Jt } f_{j,m},
\]
we have $\norm{f_m}_1\ge 1$ for all $m\in\Mt$. Hence, there is a constant $\delta=\delta(K,\LL)$ such that
\[
\max_{j\in\Jt} \abs{f_{j,m}}\ge \delta, \quad m\in\Mt.
\]
For each $m\in\Mt$ pick $j(m)\in\Jt$ such that $ \abs{f_{j(m),m}} \ge \delta$. Put
\[
\ZB=(\zz_m)_{m\in\Mt}=\enpar{L_{j(m)} ( P_{j(m)}(\yy_m))}_{m\in\Mt}.
\]
We have $\abs{\zz_m} \le \abs{\yy_m}$ and $\abs{\yy_m^*(\zz_m)}\ge \delta$ for all $m\in\Mt$. By Lemma~\ref{thm:LTImproved}, $\YB$ is equivalent to the unit vector system of $\LL_\ZB$.

Let $(\Mt_j)_{j\in \Jt}$ be the partition of $\Mt$ defined by $m\in\Mt_j$ if $j(m)=j$. If we set
\[
\ZB_j=(P_j(\yy_m))_{m\in\Mt_j},
\]
then $\LL_\ZB$ is lattice isometric to $(\oplus_{j\in \Jt} \LL_{\ZB_j})_\LL$. Since
\[
M^{(2)}(\XX_j)=M_{(2)}(\XX_j)=1,\quad j\in \Jt,
\]
we have
\[
\norm {(a_m)_{m\in\Mt_j}}_{\ZB_j}
=\enpar{\sum_{m\in\Mt_j} \abs{a_m}^2 \norm{P_j(\yy_m)}^2}^{1/2}
\]
for all $(a_m)_{m\in\Mt_j}\in c_{00}(\Mt_j)$.
Pick $j\in\Jt$ and $m\in\Mt$. Then, $\norm{P_j(\yy_m)} \le \norm{\yy_m} = 1$ and, by \eqref{eq:LargeCoef},
\[
\delta \le f_{j,m} \le K \norm{S} \norm{P_j(\yy_m)}.
\]
We infer that $\LL_\ZB$ is lattice isomorphic to $\enpar{ \bigoplus_{j\in\Jt} \ell_2^{\abs{\Mt_j}}}_\LL$.

We conclude the proof by estimating the size of $\Mt_j$. Pick $j\in\Jt$. Taking into account Lemma~\ref{lem:bilinear} we obtain
\begin{align*}
\abs{\Mt_j} &\le \sum_{m\in\Mt_j} \frac{\abs{f_{j,m}}}{\delta}\\
&=\frac{1}{\delta} \sum_{m\in\Mt_j} \abs{\sum_{n=1}^{\varphi(j)} \ee^*_{n,j} (\yy_m) \yy_m^*(\ee_{n,j})}\\
&\le \sum_{n=1}^{\varphi(j)} \sum_{m\in\Mt_j} \abs{\ee^*_{n,j} (\yy_m) \yy_m^*(\ee_{n,j})}\\
&\le \sum_{n=1}^{\varphi(j)} K \norm{\ee_{n,j}} \norm{\ee_{n,j}^*}=K\varphi(j).\qedhere
\end{align*}
\end{proof}

\begin{corollary}\label{cor:main}
Let $\LL$ be a strongly absolute sequence space.
\begin{enumerate}[label=(\roman*), leftmargin=*, widest=ii]
\item $\LL(\ell_2)$ has unique unconditional basis.
\item If $\LL$ is square-stable then $(\oplus_{n=1}^\infty \ell_2^{\varphi(n)})_\LL$ has unique unconditional basis for each non-decreasing map $\varphi\colon\NN\to\NN$.
\end{enumerate}
\end{corollary}

\begin{proof}
Just combine Theorem~\ref{thm:Ll2PU} with Lemma~\ref{lem:PUvsUTAP} and Proposition~\ref{prop:DSSquare}.
\end{proof}

We close the paper with the proof of the advertised Theorem~\ref{thm:main}.

\begin{proposition}\label{prop:TLSS}
The unit vector system of $\TL_{p,q}$, $0<p<\infty$, $0<q\le\infty$, is square stable.
\end{proposition}

\begin{proof}
Let $\pi\colon \Dt\times\{1,2\}$ be a map so that $\pi$ is a bijection from $\Dt_j\times\{1,2\}$ onto $\Dt_{j+1}$ for each $j\in\NN_0$. Consider also a map $\sigma\colon \Dt \to \Dt$ such that $\sigma$ is one-to-one from $\Dt_j$ into $\Dt_{j+1}$ for each $j\in\NN_0$.

Let $(J_n)_{n=1}^\infty$ be a natural ordering of $\Dt$. If we define $\sigma_0\colon\NN_0\to \NN$ by $\sigma_0(0)=1$ and $J_{\sigma_0(n)}=\sigma(J_n)$ for all $n\in\NN$, then $S_{\sigma_0}$ is a an isomorphic embedding from $\TL_{r,q}$ into its sublattice $\LL$ consisting of sequences whose first coordinate is null. In turn, if we define $\pi_0 \colon \NN\times\{1,2\} \to \NN$ by $J_{\pi_0(n,\delta)} = \pi(J_n,\delta)$, $S_{\pi_0}$ is an isomorphism from $\LL^2$ onto $\LL$. Hence, $\pi_0\circ(\sigma_0,\sigma_0)$ induces an isomorphic embedding.
\end{proof}

\begin{theorem}\label{thm:maingeneral}
Let $0<p<\infty$ and $0<q\le \infty$. Then $\TL_{p,q}(\ell_2)$ has a unique unconditional basis. Besides, if $\varphi\colon\Dt\to\NN$ satisfies $\varphi(J)\le \varphi(K)$ whenever $\abs{K} \le \abs{J}$, then $\enpar{\oplus_{n=0}^\infty \ell_2^{\varphi(J_n)}}_{\TL_{p,q}}$, where $J_0=I$ and $(J_n)_{n=1}^\infty$ is a natural ordering of $\Dt$, also has a unique unconditional basis.
\end{theorem}

\begin{proof}
The lattice $\Ts_{p,2}$ is $p$-convex, whence $L$-convex. Besides, it is strongly absolute by \cite{AlbiacAnsorena2022c}*{Proposition 7.5} and \eqref{eq:demlp}. Combining these results with Proposition~\ref{prop:TLSS} and Corollary~\ref{cor:main} gives the desired conclusion.
\end{proof}

\begin{proof}[Proof of Theorem~\ref{thm:main}]
Just combine Theorem~\ref{thm:maingeneral} with Theorem~\ref{thm:HpDTL}.
\end{proof}

We remark that $\Ts_{p,q}(\ell_1)$ and $\Ts_{p,q}(c_0)$, $0<p<1$, $0<q\le \infty$, also have a unique unconditional basis by \cite{AlbiacAnsorena2022c}*{Theorem 4.2} (we replace the spaces with their separable parts if $q=\infty$). If $q=2$ it is fair to call these spaces Hardy spaces. However, their connection with other classes of Hardy spaces is unclear.


\begin{bibdiv}
\begin{biblist}

\bib{AA2013}{article}{
author={Albiac, Fernando},
author={Ansorena, Jos\'{e}~L.},
title={Integration in quasi-{B}anach spaces and the fundamental theorem
of calculus},
date={2013},
ISSN={0022-1236},
journal={J. Funct. Anal.},
volume={264},
number={9},
pages={2059\ndash 2076},
url={https://doi.org/10.1016/j.jfa.2013.02.003},
review={\MR{3029146}},
}

\bib{AlbiacAnsorena2022b}{article}{
author={Albiac, Fernando},
author={Ansorena, Jos\'{e}~L.},
title={On the permutative equivalence of squares of unconditional
bases},
date={2022},
ISSN={0001-8708},
journal={Adv. Math.},
volume={410},
pages={Paper No. 108695, 26},
url={https://doi-org/10.1016/j.aim.2022.108695},
review={\MR{4487973}},
}

\bib{AlbiacAnsorena2022c}{article}{
author={Albiac, Fernando},
author={Ansorena, Jos{\'e}~L.},
title={Uniqueness of unconditional basis of infinite direct sums of
quasi-{B}anach spaces},
date={2022},
ISSN={1385-1292},
journal={Positivity},
volume={26},
number={2},
pages={Paper No. 35, 43},
url={https://doi-org/10.1007/s11117-022-00905-1},
review={\MR{4400173}},
}

\bib{AlbiacAnsorena2023}{article}{
author={Albiac, Fernando},
author={Ansorena, Jos\'{e}~L.},
title={The structure of greedy-type bases in {T}sirelson's space and its
convexifications},
date={2023},
journal={Annali della Scuola Normale Superiore di Pisa, Classe di
Scienze},
eprint={https://doi.org/10.2422/2036-2145.202202_001},
url={https://doi.org/10.2422/2036-2145.202202_001},
}

\bib{AABW2021}{article}{
author={Albiac, Fernando},
author={Ansorena, Jos\'{e}~L.},
author={Bern\'{a}, Pablo~M.},
author={Wojtaszczyk, Przemys{\l}aw},
title={Greedy approximation for biorthogonal systems in quasi-{B}anach
spaces},
date={2021},
journal={Dissertationes Math. (Rozprawy Mat.)},
volume={560},
pages={1\ndash 88},
}

\bib{AACD2018}{article}{
author={Albiac, Fernando},
author={Ansorena, Jos\'{e}~L.},
author={C\'{u}th, Marek},
author={Doucha, Michal},
title={Lipschitz free {$p$}-spaces for {$0 < p < 1$}},
date={2020},
ISSN={0021-2172},
journal={Israel J. Math.},
volume={240},
number={1},
pages={65\ndash 98},
url={https://doi-org/10.1007/s11856-020-2061-5},
review={\MR{4193127}},
}

\bib{AAW2021c}{article}{
author={Albiac, Fernando},
author={Ansorena, Jos\'{e}~L.},
author={Wojtaszczyk, Przemys\l~aw},
title={On a `philosophical' question about {B}anach envelopes},
date={2021},
ISSN={1139-1138},
journal={Rev. Mat. Complut.},
volume={34},
number={3},
pages={747\ndash 759},
url={https://doi-org/10.1007/s13163-020-00374-8},
review={\MR{4302240}},
}

\bib{AAW2022}{article}{
author={Albiac, Fernando},
author={Ansorena, Jos\'{e}~L.},
author={Wojtaszczyk, Przemys{\l}aw},
title={Uniqueness of unconditional basis of {$H_p(\mathbb{T}) \oplus
\ell_2$} and {$H_p(\mathbb{T}) \oplus\mathcal{T}^{(2)}$} for {$0 < p < 1$}},
date={2022},
ISSN={0022-1236},
journal={J. Funct. Anal.},
volume={283},
number={7},
pages={Paper No. 109597, 24},
url={https://doi-org/10.1016/j.jfa.2022.109597},
review={\MR{4447769}},
}

\bib{AlbiacKalton2016}{book}{
author={Albiac, Fernando},
author={Kalton, Nigel~J.},
title={Topics in {B}anach space theory},
edition={Second Edition},
series={Graduate Texts in Mathematics},
publisher={Springer, [Cham]},
date={2016},
volume={233},
ISBN={978-3-319-31555-3; 978-3-319-31557-7},
url={https://doi.org/10.1007/978-3-319-31557-7},
note={With a foreword by Gilles Godefroy},
review={\MR{3526021}},
}

\bib{AKL2004}{article}{
author={Albiac, Fernando},
author={Kalton, Nigel~J.},
author={Ler\'{a}noz, Camino},
title={Uniqueness of the unconditional basis of {$l_1(l_p)$} and
{$l_p(l_1)$}, {$0<p<1$}},
date={2004},
ISSN={1385-1292},
journal={Positivity},
volume={8},
number={4},
pages={443\ndash 454},
url={https://doi-org/10.1007/s11117-003-8542-z},
review={\MR{2117671}},
}

\bib{AlbiacLeranoz2002}{article}{
author={Albiac, Fernando},
author={Ler\'{a}noz, Camino},
title={Uniqueness of unconditional basis of {$l_p(c_0)$} and
{$l_p(l_2),\ 0<p<1$}},
date={2002},
ISSN={0039-3223},
journal={Studia Math.},
volume={150},
number={1},
pages={35\ndash 52},
url={https://doi.org/10.4064/sm150-1-4},
review={\MR{1893423}},
}

\bib{AnsorenaBello2022}{article}{
author={Ansorena, Jos\'{e}~L.},
author={Bello, Glenier},
title={Toward an optimal theory of integration for functions taking
values in quasi-{B}anach spaces},
date={2022},
ISSN={1578-7303},
journal={Rev. R. Acad. Cienc. Exactas F\'{\i}s. Nat. Ser. A Mat. RACSAM},
volume={116},
number={2},
pages={Paper No. 85, 38},
url={https://doi-org/10.1007/s13398-022-01230-8},
review={\MR{4396837}},
}

\bib{Aoki1942}{article}{
author={Aoki, Tosio},
title={Locally bounded linear topological spaces},
date={1942},
ISSN={0369-9846},
journal={Proc. Imp. Acad. Tokyo},
volume={18},
pages={588\ndash 594},
url={http://projecteuclid.org/euclid.pja/1195573733},
review={\MR{14182}},
}

\bib{BenLin2000}{book}{
author={Benyamini, Y.},
author={Lindenstrauss, Joram},
title={Geometric nonlinear functional analysis. {V}ol. 1},
series={American Mathematical Society Colloquium Publications},
publisher={American Mathematical Society, Providence, RI},
date={2000},
volume={48},
ISBN={0-8218-0835-4},
review={\MR{1727673}},
}

\bib{Blasco1988}{article}{
author={Blasco, Oscar},
title={Boundary values of functions in vector-valued {H}ardy spaces and
geometry on {B}anach spaces},
date={1988},
ISSN={0022-1236},
journal={J. Funct. Anal.},
volume={78},
number={2},
pages={346\ndash 364},
url={https://doi.org/10.1016/0022-1236(88)90123-1},
review={\MR{943502}},
}

\bib{Blasco1990}{incollection}{
author={Blasco, Oscar},
title={Spaces of vector valued analytic functions and applications},
date={1990},
booktitle={Geometry of {B}anach spaces ({S}trobl, 1989)},
series={London Math. Soc. Lecture Note Ser.},
volume={158},
publisher={Cambridge Univ. Press, Cambridge},
pages={33\ndash 48},
review={\MR{1110184}},
}

\bib{BGC1987}{article}{
author={Blasco, Oscar},
author={Garc\'{\i}a-Cuerva, Jos\'{e}},
title={Hardy classes of {B}anach-space-valued distributions},
date={1987},
ISSN={0025-584X,1522-2616},
journal={Math. Nachr.},
volume={132},
pages={57\ndash 65},
url={https://doi.org/10.1002/mana.19871320105},
review={\MR{910043}},
}

\bib{Bourgain1982}{article}{
author={Bourgain, J.},
title={The nonisomorphism of {$H\sp{1}$}-spaces in one and several
variables},
date={1982},
ISSN={0022-1236},
journal={J. Functional Analysis},
volume={46},
number={1},
pages={45\ndash 57},
url={https://doi.org/10.1016/0022-1236(82)90043-X},
review={\MR{654464}},
}

\bib{BCLT1985}{article}{
author={Bourgain, Jean},
author={Casazza, Peter~G.},
author={Lindenstrauss, Joram},
author={Tzafriri, Lior},
title={Banach spaces with a unique unconditional basis, up to
permutation},
date={1985},
ISSN={0065-9266},
journal={Mem. Amer. Math. Soc.},
volume={54},
number={322},
pages={iv+111},
url={https://doi-org/10.1090/memo/0322},
review={\MR{782647}},
}

\bib{BR1980}{article}{
author={Bourgain, Jean},
author={Rosenthal, Haskell~P.},
title={Martingales valued in certain subspaces of {$L\sp{1}$}},
date={1980},
ISSN={0021-2172},
journal={Israel J. Math.},
volume={37},
number={1-2},
pages={54\ndash 75},
url={https://doi.org/10.1007/BF02762868},
review={\MR{599302}},
}

\bib{BG1970}{article}{
author={Burkholder, D.~L.},
author={Gundy, R.~F.},
title={Extrapolation and interpolation of quasi-linear operators on
martingales},
date={1970},
ISSN={0001-5962,1871-2509},
journal={Acta Math.},
volume={124},
pages={249\ndash 304},
url={https://doi.org/10.1007/BF02394573},
review={\MR{440695}},
}

\bib{CZ1952}{article}{
author={Calder\'{o}n, A.~P.},
author={Zygmund, A.},
title={On the existence of certain singular integrals},
date={1952},
ISSN={0001-5962,1871-2509},
journal={Acta Math.},
volume={88},
pages={85\ndash 139},
url={https://doi.org/10.1007/BF02392130},
review={\MR{52553}},
}

\bib{Carleson1980}{article}{
author={Carleson, Lennart},
title={An explicit unconditional basis in {$H\sp{1}$}},
date={1980},
ISSN={0007-4497},
journal={Bull. Sci. Math. (2)},
volume={104},
number={4},
pages={405\ndash 416},
review={\MR{602408}},
}

\bib{CasKal1999}{article}{
author={Casazza, Peter~G.},
author={Kalton, Nigel~J.},
title={Uniqueness of unconditional bases in {$c_0$}-products},
date={1999},
ISSN={0039-3223},
journal={Studia Math.},
volume={133},
number={3},
pages={275\ndash 294},
review={\MR{1687211}},
}

\bib{CW1977}{article}{
author={Coifman, Ronald~R.},
author={Weiss, Guido},
title={Extensions of {H}ardy spaces and their use in analysis},
date={1977},
ISSN={0002-9904},
journal={Bull. Amer. Math. Soc.},
volume={83},
number={4},
pages={569\ndash 645},
url={https://doi.org/10.1090/S0002-9904-1977-14325-5},
review={\MR{447954}},
}

\bib{DHK2006}{article}{
author={Dilworth, Stephen~J.},
author={Hoffmann, Mark},
author={Kutzarova, Denka},
title={Non-equivalent greedy and almost greedy bases in {$l_p$}},
date={2006},
ISSN={0972-6802},
journal={J. Funct. Spaces Appl.},
volume={4},
number={1},
pages={25\ndash 42},
url={https://doi-org/10.1155/2006/368648},
review={\MR{2194634}},
}

\bib{EdWo1976}{article}{
author={Edelstein, I.~S.},
author={Wojtaszczyk, Przemys{\l}aw},
title={On projections and unconditional bases in direct sums of {B}anach
spaces},
date={1976},
ISSN={0039-3223},
journal={Studia Math.},
volume={56},
number={3},
pages={263\ndash 276},
url={https://doi-org/10.4064/sm-56-3-263-276},
review={\MR{425585}},
}

\bib{GCRF1985}{book}{
author={Garc\'{\i}a-Cuerva, Jos\'{e}},
author={Rubio~de Francia, Jos\'{e}~L.},
title={Weighted norm inequalities and related topics},
series={North-Holland Mathematics Studies},
publisher={North-Holland Publishing Co., Amsterdam},
date={1985},
volume={116},
ISBN={0-444-87804-1},
note={Notas de Matem\'{a}tica, 104. [Mathematical Notes]},
review={\MR{807149}},
}

\bib{Kalton1984b}{article}{
author={Kalton, Nigel~J.},
title={Convexity conditions for nonlocally convex lattices},
date={1984},
ISSN={0017-0895},
journal={Glasgow Math. J.},
volume={25},
number={2},
pages={141\ndash 152},
url={https://doi-org/10.1017/S0017089500005553},
review={\MR{752808}},
}

\bib{KLW1990}{article}{
author={Kalton, Nigel~J.},
author={Ler\'{a}noz, Camino},
author={Wojtaszczyk, Przemys{\l}aw},
title={Uniqueness of unconditional bases in quasi-{B}anach spaces with
applications to {H}ardy spaces},
date={1990},
ISSN={0021-2172},
journal={Israel J. Math.},
volume={72},
number={3},
pages={299\ndash 311 (1991)},
url={https://doi.org/10.1007/BF02773786},
review={\MR{1120223}},
}

\bib{Kinchine1923}{article}{
author={Khintchine, Aleksandr~Y.},
title={{\"U}ber dyadische {B}r{\"u}che},
date={1923},
ISSN={0025-5874},
journal={Math. Z.},
volume={18},
number={1},
pages={109\ndash 116},
url={https://doi-org/10.1007/BF01192399},
review={\MR{1544623}},
}

\bib{Leranoz1992}{article}{
author={Ler\'{a}noz, Camino},
title={Uniqueness of unconditional bases of {$c_0(l_p),\;0<p<1$}},
date={1992},
ISSN={0039-3223},
journal={Studia Math.},
volume={102},
number={3},
pages={193\ndash 207},
review={\MR{1170550}},
}

\bib{LinTza1979}{book}{
author={Lindenstrauss, Joram},
author={Tzafriri, Lior},
title={Classical {B}anach spaces. {II} -- function spaces},
series={Ergebnisse der Mathematik und ihrer Grenzgebiete [Results in
Mathematics and Related Areas]},
publisher={Springer-Verlag, Berlin-New York},
date={1979},
volume={97},
ISBN={3-540-08888-1},
review={\MR{540367}},
}

\bib{Maurey1974}{incollection}{
author={Maurey, Bernard},
title={Type et cotype dans les espaces munis de structures locales
inconditionnelles},
date={1974},
booktitle={S\'{e}minaire {M}aurey-{S}chwartz 1973--1974: {E}spaces
{$L^{p}$}, applications radonifiantes et g\'{e}om\'{e}trie des espaces de
{B}anach, {E}xp. {N}os. 24 et 25},
pages={25},
review={\MR{0399796}},
}

\bib{Maurey1980}{article}{
author={Maurey, Bernard},
title={Isomorphismes entre espaces {$H_{1}$}},
date={1980},
ISSN={0001-5962},
journal={Acta Math.},
volume={145},
number={1-2},
pages={79\ndash 120},
url={https://doi.org/10.1007/BF02414186},
review={\MR{586594}},
}

\bib{MN2002}{article}{
author={Michalak, Artur},
author={Nawrocki, Marek},
title={Banach envelopes of vector valued {$H^p$} spaces},
date={2002},
ISSN={0019-3577,1872-6100},
journal={Indag. Math. (N.S.)},
volume={13},
number={2},
pages={185\ndash 195},
url={https://doi.org/10.1016/S0019-3577(02)80004-0},
review={\MR{2016337}},
}

\bib{Rolewicz1957}{article}{
author={Rolewicz, Stefan},
title={On a certain class of linear metric spaces},
date={1957},
journal={Bull. Acad. Polon. Sci. Cl. III.},
volume={5},
pages={471\ndash 473, XL},
review={\MR{0088682}},
}

\bib{Tozoni1995}{article}{
author={Tozoni, Sergio~Antonio},
title={Weighted inequalities for vector operators on martingales},
date={1995},
ISSN={0022-247X,1096-0813},
journal={J. Math. Anal. Appl.},
volume={191},
number={2},
pages={229\ndash 249},
}

\bib{Woj1984}{article}{
author={Wojtaszczyk, Przemys{\l}aw},
title={{$H_{p}$}-spaces, {$p\leq 1$}, and spline systems},
date={1984},
ISSN={0039-3223},
journal={Studia Math.},
volume={77},
number={3},
pages={289\ndash 320},
url={https://doi-org/10.4064/sm-77-3-289-320},
review={\MR{745285}},
}

\bib{Wojtowicz1988}{article}{
author={W\'{o}jtowicz, Marek},
title={On the permutative equivalence of unconditional bases in
{$F$}-spaces},
date={1988},
ISSN={0208-6573},
journal={Funct. Approx. Comment. Math.},
volume={16},
pages={51\ndash 54},
review={\MR{965366}},
}

\end{biblist}
\end{bibdiv}

\end{document}